\documentclass{article}
\usepackage{amsmath, amsfonts, amsthm, amssymb}
\usepackage{graphicx}
\usepackage{float}
\usepackage{verbatim}

\numberwithin{equation}{section}
\newtheorem{thm}{Theorem}[section]
\newtheorem{lma}[thm]{Lemma}
\newtheorem{cor}[thm]{Corollary}
\newtheorem{defn}[thm]{Definition}

\newtheorem{prop}[thm]{Proposition}

\newtheorem{ques}[thm]{Question}

\newtheorem{exam}[thm]{Example}

\renewcommand{\geq}{\geqslant}
\renewcommand{\leq}{\leqslant}
\renewcommand{\H}{\text{H}}

\allowdisplaybreaks

\title{Assouad type dimensions and homogeneity of fractals}

\author{Jonathan M. Fraser\\ \\
\emph{Mathematical Institute, University of St Andrews, North Haugh,}\\ \emph{St Andrews, Fife, KY16 9SS, Scotland}\\ \emph{e-mail: jmf32@st-andrews.ac.uk}}

\begin{document}
\maketitle

\begin{abstract}
We investigate several aspects of the Assouad dimension and the lower dimension, which together form a natural `dimension pair'.  In particular, we compute these dimensions for certain classes of self-affine sets and quasi-self-similar sets and study their relationships with other notions of dimension, like the Hausdorff dimension for example. We also investigate some basic properties of these dimensions including their behaviour regarding unions and products and their set theoretic complexity.\\


\end{abstract}

\section{Introduction} \label{intro}

In this paper we conduct a detailed study of the Assouad dimension and the lower dimension (sometimes referred to as the \emph{minimal dimensional number}, \emph{lower Assouad dimension} or \emph{uniformity dimension}).  In particular, we investigate to what extent these dimensions are useful tools for studying the homogeneity of fractal sets.  Roughly speaking, the Assouad dimension depends only on the most complex part of the set and the lower dimension depends only on the least complex part of the set.  As such they give coarse, easily interpreted, geometric information about the extreme behaviour of the local geometry of the set and their relationship with each other and the Hausdorff, box and packing dimensions, helps paint a complete picture of the local scaling laws.  We begin with a thorough investigation of the basic properties of the Assouad and lower dimensions, for example, how they behave under products and the set theoretic complexity of the Assouad and lower dimensions as maps on spaces of compact sets.  We then compute the Assouad and lower dimensions for a wide variety of sets, including the quasi-self-similar sets of Falconer and McLaughlin \cite{implicit, mclaughlin} and
the self-affine carpets of Bara\'nski \cite{baranski} and Lalley-Gatzouras \cite{lalley-gatz}.  We also provide an example of a self-similar set with overlaps which has distinct upper box dimension and Assouad dimension, thus answering a question posed by Olsen \cite[Question 1.3]{olsenassouad}. In Section \ref{questions} we discuss our results and pose several open questions.  All of our proofs are given in Sections \ref{proofs1}--\ref{proofs3}.
\\ \\
We use various different techniques to compute the Assouad and lower dimensions.  In particular, to calculate the dimensions of self-affine carpets we use a combination of delicate covering arguments and the construction of appropriate `tangents'.  We use the notion of \emph{weak tangents} used by Mackay and Tyson \cite{mackay, mackaytyson} and \emph{very weak tangents}, a concept we introduce here specifically designed to estimate the lower dimension.

\subsection{Assouad dimension and lower dimension} \label{assouadintro}

The Assouad dimension was introduced by Assouad in the 1970s \cite{assouadphd, assouad}, see also \cite{larman}.  Let $(X,d)$ be a metric space and for any non-empty subset $F \subseteq X$ and $r>0$, let $N_r (F)$ be the smallest number of open sets with diameter less than or equal to $r$ required to cover $F$.  The \emph{Assouad dimension} of a non-empty subset $F$ of $X$, $\dim_\text{A} F$, is defined by
\begin{eqnarray*}
\dim_\text{A} F \ = \  \inf \Bigg\{ &\alpha& : \text{     there exists constants $C, \, \rho>0$ such that,} \\
&\,& \text{ for all $0<r<R\leq \rho$, we have $\ \sup_{x \in F} \, N_r\big( B(x,R) \cap F \big) \ \leq \ C \bigg(\frac{R}{r}\bigg)^\alpha$ } \Bigg\}.
\end{eqnarray*}
Although interesting in its own right, the importance of the Assouad dimension thus far has been its relationship with quasi-conformal mappings and embeddability problems, rather than as a tool in the dimension theory of fractals, see \cite{hein, luk, mackaytyson, robinson}.  However, this seems to be changing, with several recent papers appearing which study Assouad dimension and its relationship with the other well-studied notions of dimension: Hausdorff, packing and box dimension; see, for example, \cite{kaenmakiassouad, mackay, olsenassouad, SSass}.  We will denote the Hausdorff, packing and lower and upper box dimensions by $\dim_\text{H}$, $\dim_\text{P}$, $\underline{\dim}_\text{B}$ and $\overline{\dim}_\text{B}$, respectively, and if the upper and lower box dimensions are equal then we refer to the common value as the box dimension and denote it by $\dim_\text{B}$.  We will also write $\mathcal{H}^s$ for the $s$-dimensional Hausdorff measure for $s \geq 0$.  For a review of these other notions of dimension and measure, see \cite{falconer}.  We will also be concerned with the natural dual to Assouad dimension, which we call the \emph{lower dimension}.  The lower dimension of $X$, $\dim_\text{L} X$, is defined by
\begin{eqnarray*}
\dim_\text{L} X \ = \  \sup \Bigg\{ &\alpha& : \text{     there exists constants $C, \, \rho>0$ such that,} \\
&\,& \text{ for all $0<r<R\leq \rho$, we have $\ \inf_{x \in X} \, N_r\big( B(x,R)\big) \ \geq \ C \bigg(\frac{R}{r}\bigg)^\alpha$ } \Bigg\}.
\end{eqnarray*}

This quantity was introduced by Larman \cite{larman}, where it was called the \emph{minimal dimensional number}, but it has been referred to by other names, for example: the \emph{lower Assouad dimension} by K\"aenm\"aki, Lehrb\"ack and Vuorinen \cite{kaenmakiassouad} and the \emph{uniformity dimension} (Tuomas Sahlsten, personal communication). We decided on \emph{lower dimension} to be consistent with the terminology used by Bylund and Gudayol in \cite{bylund}, but we wish to emphasise the relationship with the well-studied and popular Assouad dimension.  Indeed, the Assouad dimension and the lower dimension often behave as a pair, with many of their properties being intertwined.  The lower dimension has received little attention in the literature on fractals, however, we believe it is a very natural definition and should have a place in the study of dimension theory and fractal geometry.  We summarise the key reasons for this below:
\begin{itemize}
\item The lower dimension is a natural dual to the well-studied Assouad dimension and dimensions often come in pairs.  For example, the rich and complex interplay between Hausdorff dimension and packing dimension has become one of the key concepts in dimension theory.  Also, the popular upper and lower box dimensions are a natural `dimension pair'.  Dimension pairs are important in several areas of geometric measure theory, for example, the dimension theory of product spaces, see Theorem \ref{lowerassproduct} and the discussion preceding it.

\item The lower dimension gives some important and easily interpreted information about the fine structure of the set.  In particular, it identifies the parts of the set which are easiest to cover and gives a rigorous gauge on how efficiently the set can be covered in these areas.

\item One might argue that the lower dimension is not a sensible tool for studying sets which are highly inhomogeneous in the sense of having some exceptional points around which the set is distributed very sparsely compared to the rest of the set.  For example, sets with isolated points have lower dimension equal to zero.  However, it is perfect for studying attractors of iterated function systems (IFSs) as the IFS construction forces the set to have a certain degree of homogeneity.  In fact the difference between the Assouad dimension and the lower dimension can give insight into the amount of homogeneity present.  For example, for self-similar sets satisfying the open set condition, the two quantities are equal indicating that the set is as homogeneous as possible.  However, in this paper we will demonstrate that for more complicated \emph{self-affine sets} and \emph{self-similar sets with overlaps}, the quantities can be, and often are, different.
\end{itemize}

For a totally bounded subset $F$ of a metric space, we have
\[
\dim_\text{L} F  \ \leq \  \underline{\dim}_\text{B} F \ \leq \  \overline{\dim}_\text{B} F \ \leq \  \dim_\text{A} F.
\]
The lower dimension is in general not comparable to the Hausdorff dimension or packing dimension.  However, if $F$ is compact, then
\[
\dim_\text{L} F  \ \leq \ \dim_\text{H} F \ \leq \ \dim_\text{P} F.
\]
This was proved by Larman \cite{larman, larman2}.   In particular, this means that the lower dimension provides a practical way of estimating the Hausdorff dimension of compact sets from below, which is often a difficult problem.  The Assouad dimension and lower dimensions are much more sensitive to the local structure of the set around particular points, whereas the other dimensions give more global information.  The Assouad dimension will be `large' relative to the other dimensions if there are points around which the set is `abnormally difficult' to cover and the lower dimension will be `small' relative to the other dimensions if there are points around which the set is `abnormally easy' to cover.  This phenomena is best illustrated by an example.  Let $X=\{1/n : n \in \mathbb{N}\} \cup \{0\}$.  Then
\[
\dim_\text{L} X \ = \ 0,
\]
\[
 \underline{\dim}_\text{B} X \ =\  \overline{\dim}_\text{B} X \ = \ 1/2
\]
and
\[
 \dim_\text{A} X = 1.
\]
The lower dimension is zero due to the influence of the isolated points in $X$.  Indeed the set is locally very easy to cover around isolated points and it follows that if a set, $X$, has any isolated points, then $\dim_\text{L} X =  0$.  This could be viewed as an undesirable property for a `dimension' to have because it causes it to be non-monotone and means that it can increase under Lipschitz mappings.  We are not worried by this, however, as the geometric interpretation is clear and useful.  For some basic properties of the Assouad dimension, the reader is referred to the appendix of \cite{luk} and for some more discussion of the basic properties of the Assouad and lower dimension, see Section \ref{app} of this paper.
\\ \\
It is sometimes useful to note that we can replace $N_r$ in the definition of the Assouad and lower dimensions with any of the standard covering or packing functions, see \cite[Section 3.1]{falconer}.  For example, if $F$ is a subset of Euclidean space, then $N_r(F)$ could denote the number of squares in an $r$-mesh orientated at the origin which intersect $F$ or the maximum number of sets in an $r$-packing of $F$.  We also obtain equivalent definitions if the ball $B(x,R)$ is taken to be open or closed, although we usually think of it as being closed.

\subsection{Self-affine carpets} \label{affineintro}

Since Bedford-McMullen carpets were introduced in the mid 80s \cite{bedford, mcmullen}, there has been an enormous interest in investigating these sets as well as several generalisations.  Particular attention has been paid to computing the dimensions of an array of different classes of self-affine carpets, see \cite{baranski, bedford, fengaffine, me_box, lalley-gatz, mackay, mcmullen}.  One reason that these classes are of interest is that they provide examples of self-affine sets with distinct Hausdorff and packing dimensions.
\\ \\
Recently, Mackay \cite{mackay} computed the Assouad dimension for the Lalley-Gatzouras class, see \cite{lalley-gatz}, which contains the Bedford-McMullen class.  In this paper we will compute the Assouad dimension and lower dimension for the Bara\'nski class, see \cite{baranski}, which also contains the Bedford-McMullen class, and we will complement Mackay's result by computing the lower dimension for the Lalley-Gatzouras class.  We will now briefly recall the definitions.
\\ \\
Both classes consist of self-affine attractors of finite contractive iterated function systems acting on the unit square $[0,1]^2$.  Recall that an iterated function system (IFS) is a finite collection of contracting self-maps on a metric space and the attractor of such an IFS, $\{S_1, \dots, S_N\}$, is the unique non-empty compact set, $F$, satisfying
\[
F = \bigcup_{i=1}^{N} S_i(F).
\]
An important class of IFSs is when the mappings are translate linear and act on Euclidean space.  In such cirmumstances, the attractor is called a self-affine set.  These sets have attracted a substantial amount of attention in the literature over the past 30 years and are generally considered to be significantly more difficult to deal with than self-similar sets, where the mappings are assumed to be similarities.
\\ \\
\textbf{Lalley-Gatzouras and extended Lalley-Gatzouras carpets:}
Take the unit square and divide it up into columns via a finite positive number of vertical lines.  Now divide each column up independently by slicing horizontally.  Finally select a subset of the small rectangles, with the only restriction being that the length of the base must be strictly greater than the height,  and for each chosen subrectangle include a map in the IFS which maps the unit square onto the rectangle via an orientation preserving linear contraction and a translation.  The Hausdorff and box-counting dimensions of the attractors of such systems were computed in \cite{lalley-gatz} and the Assouad dimension was computed in \cite{mackay}.  If we relax the requirement that `the length of the base must be strictly greater than the height' to `the length of the base must be greater than or equal the height', then we obtain a slightly more general class, which we will refer to as the \emph{extended Lalley-Gatzouras class}. 

\begin{figure}[H]
	\centering
	\includegraphics[width=110mm]{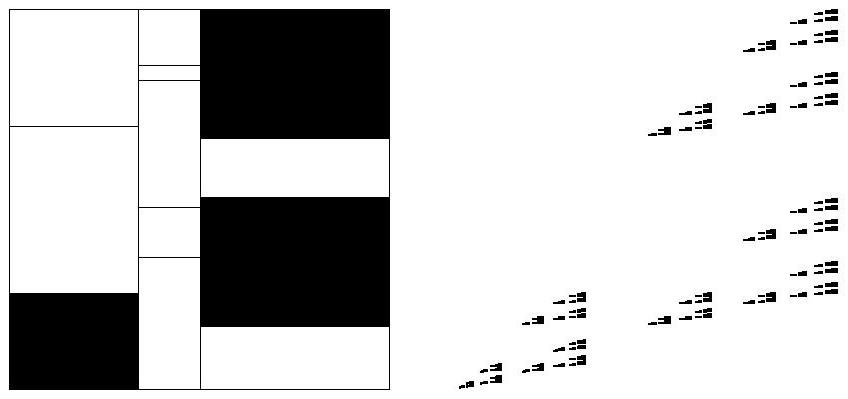}
\caption{The defining pattern for an IFS in the Lalley-Gatzouras class (left) and the corresponding attractor (right).}
\end{figure}

\textbf{Bara\'nski carpets:}
Again take the unit square, but this time divide it up into a collection of subrectangles by slicing horizontally and vertically a finite number of times (at least once in each direction).  Now take a subset of the subrectangles formed and form an IFS as above.  The key difference between the Bara\'nski and Lalley-Gatzouras classes is that in the Bara\'nski class the largest contraction need not be in the vertical direction.  This makes the Bara\'nski class significantly more difficult to deal with.  The Hausdorff and box-counting dimensions of the attractors of such systems were computed in \cite{baranski}.

\begin{figure}[H]
	\centering
	\includegraphics[width=110mm]{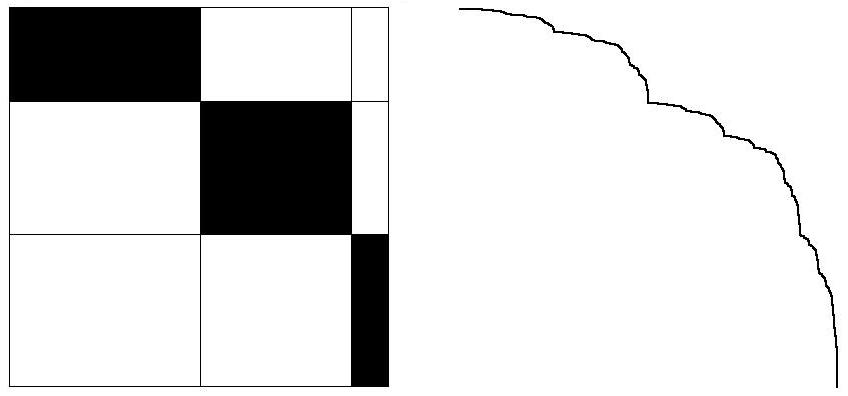}
\caption{The defining pattern for an IFS in the Bara\'nski class (left) and the corresponding attractor (right).}
\end{figure}

Note that neither class is more general than the other.  More general classes, containing both the Lalley-Gatzouras and Bara\'nski classes have been introduced and studied by Feng and Wang \cite{fengaffine} and the author \cite{me_box}.

\section{Results} \label{results}

We split this section into three parts, where we study: basic properties; quasi-self-similar sets; and self-affine sets, respectively.

\subsection{Basic properties of the Assouad and lower dimensions} \label{app}

In this section we collect together some basic results concerning the Assouad and lower dimensions.  We will be interested in how they behave under some standard set operations: unions, products and closures.  The behaviour of the classical dimensions under these operations has been long known, see \cite[Chapters 3--4]{falconer}.  We also give a simple example which demonstrates that the lower dimension of an open set in $\mathbb{R}^n$ need not be $n$.  This is in stark contrast to the rest of the dimensions.  Finally, we will investigate the measurability of the Assouad and lower dimensions.  We will frequently refer to the already known basic properties of the Assouad dimension which were due to Assouad \cite{assouadphd, assouad} and discussed in \cite{luk}.  Throughout this section $X$ and $Y$ will be metric spaces.
\\ \\
The first standard geometric construction we will consider is taking the product of two metric spaces, $(X,d_X)$ and $(Y,d_Y)$.  There are many natural `product metrics' to impose on the product space $X \times Y$, but any reasonable choice is bi-Lipschitz equivalent to the metric $d_{X \times Y}$ on $X \times Y$  defined by
\[
d_{X \times Y}\big((x_1,y_1), (x_2,y_2)\big) = \max \{d_X(x_1,x_2), \,  d_Y(y_1,y_2)\}
\]
which we will use from now on.  A classical result due to Howroyd \cite{products} is that
\[
\dim_\text{H} X \, + \, \dim_\text{H} Y \ \leq \ \dim_\text{H} (X \times Y)  \  \leq \  \dim_\text{H} X \, + \, \dim_\text{P} Y \ \leq \  \dim_\text{P} (X \times Y)  \  \leq \  \dim_\text{P} X \, + \, \dim_\text{P} Y
\]
and, indeed, it is easy to see that
\[
\underline{\dim}_\text{B} X \, + \, \underline{\dim}_\text{B} Y \  \leq \  \underline{\dim}_\text{B} (X \times Y)  \ \leq \  \underline{\dim}_\text{B} X  \, +  \, \overline{\dim}_\text{B} Y  \ \leq \  \overline{\dim}_\text{B} (X \times Y)  \  \leq  \  \overline{\dim}_\text{B} X \, + \, \overline{\dim}_\text{B} Y.
\]
In particular, `dimension pairs' are intimately related to the dimension theory of products.  Here we show that an analogous phenomenon holds for the Assouad and lower dimensions.

\begin{thm}[products] \label{lowerassproduct}
We have
\[
\dim_\text{\emph{L}} X \, + \, \dim_\text{\emph{L}} Y \  \leq \  \dim_\text{\emph{L}} (X \times Y)   \ \leq \  \dim_\text{\emph{L}} X  \, +  \, \dim_\text{\emph{A}} Y  \ \leq \  \dim_\text{\emph{A}} (X \times Y)  \  \leq  \  \dim_\text{\emph{A}} X \, + \, \dim_\text{\emph{A}} Y
\]
and
\[
\dim_\text{\emph{L}} (X^n) \ = \ n\, \dim_\text{\emph{L}} X.
\]
\end{thm}

We will prove Theorem \ref{lowerassproduct} in Section \ref{basicproductproof}.  We note that this generalises a result of Assouad, see Luukkainen \cite[Theorem A.5 (4)]{luk} and Robinson \cite[Lemma 9.7]{robinson}, which gave the following bounds for the Assoaud dimension of a product:
\[
\max\{\dim_\text{A} X, \, \dim_\text{A} Y\}  \ \leq \  \dim_\text{A} (X \times Y)  \  \leq  \  \dim_\text{A} X \, + \, \dim_\text{A} Y.
\]
The result in \cite{luk} was stated for a finite product $X_1 \times \cdots \times X_n$, but here we only give the formula for the product of two sets and note that this can easily be used to estimate the dimensions of any finite product. Also, note that the precise formula given here for the product of a metric space with itself also holds for Assouad dimension, see \cite[Theorem A.5 (4)]{luk}.  Another standard property of many dimensions is stability under taking unions (finite or countable).  Indeed, for Hausdorff, packing, upper box dimension or Assouad dimension, one has that the dimension of the union of two sets is the \emph{maximum} of the individual dimensions, with the situation for lower box dimension being more complicated.  The following proposition concerns the stability of the lower dimension.

\begin{thm}[unions]  \label{lowerassunion}
For all $E, F \subseteq X$, we have
\[
\min \{ \dim_\text{\emph{L}} E,  \ \dim_\text{\emph{L}}  F\} \ \leq \ \dim_\text{\emph{L}} ( E \cup F) \ \leq \ \max \{ \dim_\text{\emph{L}} E,  \ \dim_\text{\emph{A}}  F\}.
\]
However, if $E$ and $F$ are such that $\inf_{x \in E, \, y \in F} d(x,y)>0$, then
\[
\dim_\text{\emph{L}} ( E \cup F) = \min \{ \dim_\text{\emph{L}} E, \ \dim_\text{\emph{L}}  F\}.
\]
\end{thm}
We will prove Theorem \ref{lowerassunion} in Section \ref{basicunionproof}.  Assouad proved that the Assouad dimension is stable under taking closures, see \cite[Theorem A.5 (2)]{luk}.  Here we prove that this is also true for the lower dimension.  This does not hold for Hausdorff and packing dimension but does hold for the box dimensions.

\begin{thm}[closures]  \label{closurestable2}
For all $F \subseteq X$ we have
\[
\dim_\text{\emph{L}} F = \dim_\text{\emph{L}} \overline{F}.
\]
\end{thm}

We will prove Theorem \ref{closurestable2} in Section \ref{basicclosureproof}.  It is well known that Hausdorff, packing and box dimension cannot increase under a Lipschitz mapping.  Despite this \emph{not} being the case for either Assouad dimension or lower dimension, (see \cite[Example A.6.2]{luk} or Section \ref{selfsimexample} of this paper for the Assouad dimension case; the lower dimension case is trivial) these dimensions are still \emph{bi-Lipschitz invariant}.  For the Assouad dimension case see \cite[Theorem A.5.1]{luk} and we prove the lower dimension case here.

\begin{thm}[bi-Lipschitz invariance]  \label{bi-Lipschitz}
If $X$ and $Y$ are bi-Lipschitz equivalent, i.e., there exists a bi-Lipschitz bijection between $X$ and $Y$, then
\[
\dim_\text{\emph{L}} X = \dim_\text{\emph{L}} Y.
\]
\end{thm}

We will prove Theorem \ref{bi-Lipschitz} in Section \ref{basicbi-Lipschitzproof}.  One final standard property of the classical dimensions is that if $V \subseteq \mathbb{R}^n$ is open (or indeed has non-empty interior), then the dimension of $V$ is $n$.  This holds true for Hausdorff, box, packing and Assouad dimension, see \cite[Theorem A.5 (6)]{luk} for the Assouad dimension case.  Here we give a simple example which shows that this does \emph{not} hold for lower dimension.

\begin{exam}[open sets]
\emph{Let
\[
V \ = \  \bigcup_{n=1}^\infty \  \big(1/n-2^{-n}, \ 1/n+2^{-n}\big) \ \subseteq \  \mathbb{R}.
\]
It is clear that $V$ is open and we will now argue that $\dim_\text{L} V = 0$.  Let $s, C,\rho>0$ and observe that if we choose
\[
R(n) = \frac{1}{n(n+1)} - \frac{2}{2^{n-1}} \qquad \text{and} \qquad r(n) = \frac{2}{2^n}
\]
then for all $n \geq 9$ we have $0<r(n) < R(n)$ and
\[
N_{r(n)} \Big( B\big(1/n, R(n)\big) \cap V  \Big) = 1
\]
since $B\big(1/n, R(n) \big) \cap V = \big(1/n-2^{-n}, \ 1/n+2^{-n}\big)$.  We may hence choose $n$ large enough to ensure that $0<r(n) < R(n)<\rho$ and
\[
N_{r(n)} \Big( B\big(1/n, R(n)  \big) \cap V \Big) \ = \  1 \  <  \ C \,   \bigg(\frac{R(n)}{r(n)} \bigg)^s
\]
which gives that $\dim_\text{L} V \leq s$ and letting $s \searrow 0$ proves the result.}
\end{exam}

Finally, we will examine the measurabilty properties of the Assouad dimension and lower dimensions as functions from the compact subsets of a given compact metric space into $\mathbb{R}$.  This question has been examined thoroughly in the past for other definitions of dimension.  For example, in \cite{mattilamauldin} it was shown that Hausdorff dimension and upper and lower box dimensions are Borel measurable and, moreover, are of Baire class 2, but packing dimension is not Borel measurable.  Measurability properties of several multifractal dimension functions were also considered in \cite{multiborel}.
\\ \\
We will now briefly recall the \emph{Baire hierarchy} which is used to classify functions by their `level of discontinuity'.  Let $(A, d_A)$ and $(B, d_B)$ be metric spaces.  A function $f: A \to B$ is of Baire class 0 if it is continuous.  The later classes are defined inductively by saying that a function $f:A \to B$ is of Baire class $n+1$ if it is in the pointwise closure of the Baire class $n$ functions.  One should think of functions lying in higher Baire classes (and not in lower ones) as being `further away' from being continuous and as being more set theoretically complicated.  Also, if a function belongs to any Baire class, then it is Borel measurable.  For more details on the Baire hierarchy, see \cite{kechris}.
\\ \\
For the rest of this section let $(X,d)$ be a compact metric space and let $\mathcal{K}(X)$ denote the set of all non-empty compact subsets of $X$.  We are interested in maps from $\mathcal{K}(X)$ into $\mathbb{R}$, so we will now metricise $\mathcal{K}(X)$ in the usual way.  Define the Hausdorff metric, $d_\mathcal{H}$, by
\[
d_\mathcal{H}(E,F) = \inf \{ \varepsilon>0: E \subseteq F_\varepsilon \text{ and } F \subseteq E_\varepsilon\}
\]
for $E,F \in \mathcal{K}(X)$ and where $E_\varepsilon$ denotes the $\varepsilon$-neighbourhood of $E$.  It is sometimes convenient to extend $d_\mathcal{H}$ to a metric $d_\mathcal{H}'$ on the space $\mathcal{K}_0(X) = \mathcal{K}(X) \cup \{ \emptyset \}$, by letting $d'_\mathcal{H}(E, \emptyset) = \text{diam}(X)$ for all $E \in \mathcal{K}(X)$.  The spaces $(\mathcal{K}(X),d_\mathcal{H})$ and $(\mathcal{K}_0(X),d'_\mathcal{H})$ are both complete.  We can now state our final results of this section.  
\begin{thm} \label{meas1}
The function $\Delta_\text{\emph{A}}: \mathcal{K}(X) \to \mathbb{R}$ defined by
\[
\Delta_\text{\emph{A}}(F) = \dim_\text{\emph{A}} F
\]
is of Baire class 2 and, in particular, Borel measurable.
\end{thm}

\begin{thm} \label{meas2}
The function $\Delta_\text{\emph{L}}: \mathcal{K}(X) \to \mathbb{R}$ defined by
\[
\Delta_\text{\emph{L}}(F) = \dim_\text{\emph{L}} F
\]
is of Baire class 3 and, in particular, Borel measurable.
\end{thm}

We will prove Theorems \ref{meas1} and \ref{meas2} in Section \ref{basicmeasureproof1}.  It is straightforward to see that neither the Assouad dimension nor the lower dimension are Baire 1 as for Baire 1 functions the points of continuity form a dense $\mathcal{G}_\delta$ set, see \cite[Theorem 24.14]{kechris}, and it is evident that the Assouad and lower dimensions are discontinuous everywhere.  Hence, the Assouad dimension is `precisely' Baire 2, but we have been unable to determine the `precise' Baire class of the lower dimension.

\subsection{Dimension results for quasi-self-similar sets}

In this section we will examine sets with high degrees of homogeneity.  We will be particularly interested in conditions which guarantee the equality of certain dimensions.  Throughout this section $(X,d)$ will be a compact metric space.  Recall that $(X,d)$ is called \emph{Ahlfors regular} if $\dim_\H X<\infty$ and there exists a constant $\lambda>0$ such that, writing $\mathcal{H}^{\dim_\text{\H} X}$ to denote the Hausdorff measure in the critical dimension,
\[
\tfrac{1}{\lambda}\,  r^{\dim_\text{\H} X} \ \leq \  \mathcal{H}^{\dim_\text{\H} X} \big( B(x,r)\big) \  \leq \ \lambda \, r^{\dim_\text{\H} X} 
\]
for all $x \in X$ and all $0<r< \text{diam}(X)$, see \cite[Chapter 8]{hein}.  A metric space is called \emph{locally Ahlfors regular} if the above estimates on the measure of balls holds for sufficiently small $r>0$.  It is easy to see that a compact locally Ahlfors regular space is Ahlfors regular.   In a certain sense Ahlfors regular spaces are the most homogeneous spaces.  This is reflected in the following proposition.

\begin{prop}
If $(X,d)$ is Ahlfors regular, then
\[
\dim_\text{\emph{L}} X  \ = \ \dim_\text{\emph{A}} X.
\]
\end{prop}

For a proof of this see, for example, \cite{bylund}.  We will now consider the Assouad and lower dimensions of quasi-self-similar sets, which are a natural class of sets exhibiting a high degree of homogeneity.  We will define quasi-self-similar sets via the implicit theorems of Falconer \cite{implicit} and McLaughlin \cite{mclaughlin}.  These results allow one to deduce facts about the dimensions and measures of a set without having to calculate them explicitly.  This is done by showing that, roughly speaking, parts of the set can be `mapped around' onto other parts without too much distortion.

\begin{defn} \label{quasidef}
A non-empty compact set $F\subseteq (X,d)$ is called quasi-self-similar if there exists $a>0$ and $r_0>0$ such that the following two conditions are satisfied:
\begin{itemize}
\item[(1)] for every set $U$ that intersects $F$ with $\lvert U \rvert \leq r_0$, there is a mapping $g:F \cap U \to F$ satisfying
\[
a \, \lvert U \rvert^{-1} \, \lvert x-y \rvert \, \leq \, \lvert g(x)-g(y) \rvert \qquad \qquad (x,y \in F \cap U)
\]
\item[(2)] for every closed ball $B$ with centre in $F$ and radius $r \leq r_0$, there is a mapping $g:F \to F \cap B$ satisfying
\[
a \,r \, \lvert x-y \rvert \, \leq \, \lvert g(x)-g(y) \rvert \qquad \qquad (x,y \in F)
\]
\end{itemize}
\end{defn}

Writing $s=\dim_\text{H} F$, it was shown in \cite{mclaughlin, implicit} that condition (1) is enough to guarantee that $\mathcal{H}^s(F) \geq a^s >0$ and $\underline{\dim}_\text{B} F = \overline{\dim}_\text{B} F =s$ and it was shown in \cite{implicit} that condition (2) is enough to guarantee $\mathcal{H}^s(F) \leq 4^s\, a^{-s} < \infty$ and $\underline{\dim}_\text{B} F = \overline{\dim}_\text{B} F =s$.  Also see \cite[Chapter 3]{techniques}.  Here we extend these implicit results to include the Assouad and lower dimensions.

\begin{thm} \label{quasi}
Let $F$ be a non-empty compact subset of $X$. 
\begin{itemize}
\item[(1)] If $F$ satisfies condition (1) in the definition of quasi-self-similar, then
\[
\dim_\text{\emph{L}} F  \ \leq \ \dim_\text{\emph{H}} F \ = \ \dim_\text{\emph{P}} F \ = \  \dim_\text{\emph{B}} F  \ =  \  \dim_\text{\emph{A}} F .
\]
\item[(2)] If $F$ satisfies condition (2) in the definition of quasi-self-similar, then
\[
\dim_\text{\emph{L}} F  \ = \ \dim_\text{\emph{H}} F \ = \ \dim_\text{\emph{P}} F \ = \  \dim_\text{\emph{B}} F  \ \leq \  \dim_\text{\emph{A}} F.
\]
\item[(3)] If $F$ satisfies conditions (1) and (2) in the definition of quasi-self-similar, then we have
\[
\dim_\text{\emph{L}} F  \ = \  \dim_\text{\emph{H}} F \ = \ \dim_\text{\emph{P}} F \ = \  \dim_\text{\emph{B}} F \ = \  \dim_\text{\emph{A}} F
\]
and moreover, $F$ is Ahlfors regular.
\end{itemize}
\end{thm}

The proof of Theorem \ref{quasi} is fairly straightforward, but we defer it to Section \ref{proofquasi}.  We obtain the following corollary which gives useful relationships between the Assouad, lower and Hausdorff dimensions in a variety of contexts. 

\begin{cor}  \label{quasi2}
The following classes of sets are Ahlfors regular and, in particular, have equal Assouad and lower dimension:
\begin{itemize}
\item[(1)] self-similar sets satisfying the open set condition;
\item[(2)] graph-directed self-similar sets satisfying the graph-directed open set condition;
\item[(3)] mixing repellers of $C^{1+\alpha}$ conformal mappings on Riemann manifolds;
\item[(4)]  Bedford's recurrent sets satisfying the open set condition, see \cite{bedfordrec}.
\end{itemize}
The following classes of sets have equal Assouad dimension and Hausdorff dimension:
\begin{itemize}
\item[(5)] sub-self-similar sets satisfying the open set condition, see \cite{subselfsim};
\item[(6)] boundaries of self-similar sets satisfying the open set condition.
\end{itemize}
The following classes of sets have equal lower dimension and Hausdorff dimension regardless of separation conditions:
\begin{itemize}
\item[(7)] self-similar sets;
\item[(8)] graph-directed self-similar sets;
\item[(9)] Bedford's recurrent sets, see \cite{bedfordrec};
\end{itemize}
\end{cor}

\begin{proof}
This follows immediately from Theorem \ref{quasi} and the fact that the sets in each of the classes (1)-(4) are quasi-self-similar, see \cite{implicit}; the sets in each of the classes (5)-(6) satisfy condition (1) in the definition of quasi-self-similar, see \cite{implicit, subselfsim} and the sets in each of the classes (7)-(9) satisfy condition (2) in the definition of quasi-self-similar, see \cite{implicit}.
\end{proof}

We do not claim that all the information presented in the above corollary is new.  For example, the fact that self-similar sets satisfying the open set condition are Ahlfors regular dates back to Hutchinson, see \cite{hutch}.  Also, Olsen \cite{olsenassouad} recently gave a direct proof that graph-directed self-similar sets (more generally, graph-directed moran constructions) have equal Hausdorff dimension and Assouad dimension.  Corollary \ref{quasi2} unifies previous results and demonstrates further that sets with equal Assouad dimension and lower dimension should display a high degree of homogeneity.
\\ \\
Finally, we remark that Theorem \ref{quasi} is sharp, in that the inequalities in parts (1) and (2) cannot be replaced with equalities in general.  To see this note that the inequality in (1) is sharp as the unit interval union a single isolated point satisfies condition (1) in the definition of quasi-self-similar, but has lower dimension strictly less than Hausdorff dimension and the inequality in (2) is sharp because self-similar sets which do not satisfy the open set condition can have Assouad dimension strictly larger than Hausdorff dimension and such sets satisfy condition (2) in the definition of quasi-self-similar.  We will prove this in Section \ref{selfsimexample} by providing an example.

\subsection{Dimension results for self-affine sets}

In this section we state our main results on the Assouad and lower dimensions of self-affine sets.  In contrast to the sets considered in the previous section, self-affine sets often exhibit a high degree of inhomogeneity.  This is because the mappings can stretch by different amounts in different directions.  We will refer to a set $F$ as a \emph{self-affine carpet} if it is the attractor of an IFS in the extended Lalley-Gatzouras or Bara\'nski class, discussed in Section \ref{affineintro}, which has at least one map which is not a similarity.  Note that the reason we assume that one of the mappings is not a similarity is so that the sets are genuinely self-affine.  The dimension theory for genuinely self-affine sets is very different from self-similar sets and we intentionally keep the two classes separate, with the self-similar case dealt with in the previous section.  We will divide the class of self-affine carpets into three subclasses, \emph{horizontal}, \emph{vertical} and \emph{mixed}, which will be described below.  In order to state our results, we need to introduce some notation. Throughout this section $F$ will be a self-affine carpet which is the attractor of an IFS $\{S_i\}_{i \in \mathcal{I}}$ for some finite index set $\mathcal{I}$, with $\lvert \mathcal{I} \rvert \geq 2$.  The maps $S_i$ in the IFS will be translate linear orientation preserving contractions on $[0,1]^2$ of the form
\[
S_i \big((x,y)\big) = (c_ix, d_iy)+(a,b)
\]
for some contraction constants $c_i \in (0,1)$ in the horizontal direction and $d_i \in (0,1)$ in the vertical direction and a translation $(a,b) \in \mathbb{R}^2$.  We will say that $F$ is of \emph{horizontal type} if $c_i\geq d_i$ for all $i \in \mathcal{I}$; of \emph{vertical type} if $c_i\leq d_i$ for all $i \in \mathcal{I}$; and of \emph{mixed type} if $F$ falls into neither the horizontal or vertical classes.  We remark here that the horizontal and vertical classes are equivalent as one can just rotate the unit square by $90^\text{o}$ to move from one class to the other.  The horizontal (and hence also vertical) class is precisely the Lalley-Gatzouras class and the Bara\'nski class is split between vertical, horizontal and mixed, with carpets of mixed type being considerably more difficult to deal with and represent the major advancement of the work of Bara\'nski \cite{baranski} over the much earlier work by Lalley and Gatzouras \cite{lalley-gatz}.
\\ \\
Let $\pi_1$ denote the projection mapping from the plane to the horizontal axis and $\pi_2$ denote the projection mapping from the plane to the vertical axis. Also, for $i \in \mathcal{I}$ let
\[
\text{Slice}_{1,i}(F) = \text{the vertical slice of $F$ through the fixed point of $S_i$}
\]
and let 
\[
\text{Slice}_{2,i}(F) = \text{the horizontal slice of $F$ through the fixed point of $S_i$}.
\]
Note that the sets $\pi_1 (F)$, $\pi_2(F)$, $\text{Slice}_{1,i}(F)$ and $\text{Slice}_{2,i}(F)$ are self-similar sets satisfying the open set condition and so their box dimension can be computed via Hutchinson's formula.  We can now state our dimension results.

\begin{thm} \label{uppermain}
Let $F$ be a self-affine carpet.  If $F$ is of horizontal type, then
\[
\dim_\text{\emph{A}} F \ = \   \dim_\text{\emph{B}} \pi_1(F) \, + \,  \max_{i \in \mathcal{I}} \,  \dim_\text{\emph{B}}  \text{\emph{Slice}}_{1,i}(F);
\]
if $F$ is of vertical type, then
\[
\dim_\text{\emph{A}} F \ = \   \dim_\text{\emph{B}} \pi_2(F) \, + \,  \max_{i \in \mathcal{I}} \,  \dim_\text{\emph{B}}  \text{\emph{Slice}}_{2,i}(F);
\]
and if $F$ is of mixed type, then
\[
\dim_\text{\emph{A}} F \ = \    \max_{i \in \mathcal{I}} \, \max_{j=1,2} \, \Big(\dim_\text{\emph{B}} \pi_j(F) \, + \,  \dim_\text{\emph{B}}  \text{\emph{Slice}}_{j,i}(F) \Big).
\]
\end{thm}

We will prove Theorem \ref{uppermain} for the mixed class in Section \ref{upperbaranski} and for the horizontal and vertical classes in Section \ref{upperlalley}. If $F$ is in the (non-extended) Lalley-Gatzouras class, then the above result was obtained in \cite{mackay}.  

\begin{thm} \label{lowermain}
Let $F$ be a self-affine carpet.  If $F$ is of horizontal type, then
\[
\dim_\text{\emph{L}} F \ = \   \dim_\text{\emph{B}} \pi_1(F) \, + \,  \min_{i \in \mathcal{I}} \,  \dim_\text{\emph{B}}  \text{\emph{Slice}}_{1,i}(F);
\]
if $F$ is of vertical type, then
\[
\dim_\text{\emph{L}} F \ = \   \dim_\text{\emph{B}} \pi_2(F) \, + \,  \min_{i \in \mathcal{I}} \,  \dim_\text{\emph{B}}  \text{\emph{Slice}}_{2,i}(F);
\]
and if $F$ is of mixed type, then
\[
\dim_\text{\emph{L}} F \ = \    \min_{i \in \mathcal{I}} \, \min_{j=1,2} \, \Big(\dim_\text{\emph{B}} \pi_j(F) \, + \,  \dim_\text{\emph{B}}  \text{\emph{Slice}}_{j,i}(F) \Big).
\]
\end{thm}

We will prove Theorem \ref{lowermain} for the mixed class in Section \ref{lowerbaranski} and for the horizontal and vertical classes in Section \ref{lowerlalley}.  We remark here that the formulae presented in Theorems \ref{uppermain} and \ref{lowermain} are completely explicit and can be computed easily to any required degree of accuracy.  It is interesting to investigate conditions for which the dimensions discussed here are equal or distinct.  Mackay \cite{mackay} noted a fascinating dichotomy for the Lalley-Gatzouras class in that either the Hausdorff dimension, box dimension and Assouad dimension are all distinct or are all equal.  We obtain the following extension of this result.

\begin{cor} \label{intcor1}
Let $F$ be a self-affine carpet in the horizontal or vertical class.  Then either
\[
\dim_\text{\emph{L}} F < \dim_\text{\emph{H}} F < \dim_\text{\emph{B}} F < \dim_\text{\emph{A}} F 
\]
or
\[
\dim_\text{\emph{L}} F = \dim_\text{\emph{H}} F = \dim_\text{\emph{B}} F = \dim_\text{\emph{A}} F.
\]
\end{cor}

We will prove Corollary \ref{intcor1} in Section \ref{proofintcor1}.  It is natural to wonder if this dichotomy also holds for the mixed class.  In fact it does not and in Section \ref{examplesB} we provide an example of a self-affine set in the mixed class for which $\dim_\text{L} F < \dim_\text{H} F = \dim_\text{B} F = \dim_\text{A} F$.  We do obtain the following slightly weaker result.

\begin{cor} \label{intcor2}
Let $F$ be a self-affine carpet.  Then either
\[
\dim_\text{\emph{L}} F <  \dim_\text{\emph{B}} F
\]
or
\[
\dim_\text{\emph{L}} F = \dim_\text{\emph{H}} F = \dim_\text{\emph{B}} F = \dim_\text{\emph{A}} F.
\]
\end{cor}

We will prove Corollary \ref{intcor2} in Section \ref{proofintcor2}.  Theorems \ref{uppermain}-\ref{lowermain} provide explicit means to estimate, or at least obtain non-trivial bounds for, the Hausdorff dimension and box dimension.  The formulae for the box dimensions given in \cite{lalley-gatz, baranski} are completely explicit, but the formulae for the Hausdorff dimensions are not explicit and are often difficult to evaluate.  As such, our results concerning lower dimension provide completely explicit and easily computable lower bounds for the Hausdorff dimension.  Finally we note that, despite how apparently easy it is to have lower dimension equal to zero, it is easy to see from Theorem \ref{lowermain} that the lower dimension of a self-affine carpet is always strictly positive.

\section{Examples} \label{examples}

In this section we give two examples and compute their Assouad and lower dimensions.  Each example is designed to illustrate an important phenomenon.

\subsection{A self-similar set with overlaps} \label{selfsimexample}

Self-similar sets with overlaps are currently at the forefront of research on fractals and are notoriously difficult to deal with.  For example, a recent paper of Hochman \cite{hochman} has made a major contribution to the famous problem of when a `dimension drop' can occur, in particular, when the Hausdorff dimension of a self-similar subset of the line can be strictly less than the minimum of the similarity dimension and one.  In this section we provide an example of a self-similar set $F \subset [0,1]$ with overlaps for which
\[
\dim_\text{L} F = \dim_\text{H} F = \dim_\text{B} F  <  \dim_\text{A} F.
\]
This answers a question of Olsen \cite[Question 1.3]{olsenassouad} which asked if it was possible to find a graph-directed Moran fractal $F$ with $\dim_\text{B} F  <  \dim_\text{A} F$.  Self-similar sets are the most commonly studied class of graph-directed Moran fractals, see \cite{olsenassouad} for more details.  We also use this example to show that Assouad dimension can increase under Lipschitz maps and, in particular, projections.
\\ \\
Let $\alpha, \beta, \gamma \in (0,1)$ be such that $(\log\beta)/(\log\alpha) \notin \mathbb{Q}$ and define similarity maps $S_1, S_2, S_3$ on $[0,1]$ as follows
\[
S_1(x) = \alpha x, \qquad S_2(x) = \beta x \qquad \text{and} \qquad S_3(x) = \gamma x +(1-\gamma).
\]
Let $F$ be the self-similar attractor of $\{S_1, S_2, S_3\}$.  We will now prove that $\dim_\text{A} F = 1$ and, in particular, the Assouad dimension is independent of $\alpha, \beta, \gamma$ provided they are chosen with the above property.  We will use the following proposition due to Mackay and Tyson, see \cite[Proposition 7.1.5]{mackaytyson}.
\begin{prop}[Mackay-Tyson] \label{weaktang00}
Let $X \subset \mathbb{R}$ be compact and let $F$ be a compact subset of $X$.  Let $T_k$ be a sequence of similarity maps defined on $\mathbb{R}$ and suppose that $T_k(F) \cap X \to_{d_\mathcal{H}} \hat{F}$ for some non-empty compact set $\hat{F} \in \mathcal{K}(X)$.  Then $\dim_\text{\emph{A}} \hat{F}   \leq  \dim_\text{\emph{A}} F$.  The set $\hat{F}$ is called a \emph{weak tangent} to $F$.
\end{prop}
We will now show that $[0,1]$ is a weak tangent to $F$ in the above sense.  Let $X = [0,1]$ and assume without loss of generality that $\alpha<\beta$.  For each $k \in \mathbb{N}$ let $T_k$ be defined by
\[
T_k(x) = \beta^{-k}x.
\]
We will now show that $T_k(F) \cap [0,1] \to_{d_\mathcal{H}} [0,1]$.  Since
\[
E_k \ := \ \big\{\alpha^m\beta^n: m \in \mathbb{N}, n \in \{-k, \dots, \infty\}\big\} \cap [0,1]  \ \subset \  T_k(F) \cap [0,1]
\]
for each $k$ it suffices to show that $E_k \to_{d_\mathcal{H}} [0,1]$.  Indeed, we have
\begin{eqnarray*}
E_k &\to_{d_\mathcal{H}}& \overline{\bigcup_{k \in \mathbb{N} } E_k} \cap [0,1] \\ \\
&=& \overline{\{\alpha^m\beta^n: m \in \mathbb{N}, n \in \mathbb{Z} \}} \cap [0,1] \\ \\
&=&[0,1].
\end{eqnarray*}
It now follows from Proposition \ref{weaktang00} that $\dim_\text{A} F = 1$.  To see why $ \overline{\{\alpha^m\beta^n: m \in \mathbb{N}, n \in \mathbb{Z} \}} \cap [0,1]=[0,1]$ we apply Dirichlet's Theorem in the following way.  It suffices to show that 
\[
\{m\log\alpha+n\log\beta: m \in \mathbb{N}, n \in \mathbb{Z} \}
\]
is dense in $(-\infty, 0)$.  We have
\[
m\log\alpha+n\log\beta = n\log\alpha \bigg(\frac{m}{n}+ \frac{\log\beta}{\log\alpha}\bigg)
\]
and Dirichlet's Theorem gives that there exists infinitely many $n$ such that
\[
\Big\lvert \frac{m}{n}+ \frac{\log\beta}{\log\alpha} \Big\rvert < 1/n^2
\]
for some $m$.  Since $\log\beta/\log\alpha$ is irrational,  we may choose $m,n$ to make
\[
0<\lvert m\log\alpha+n\log\beta \rvert < \frac{\lvert \log\alpha \lvert }{n}
\]
with $n$ arbitrarily large.  We can thus make $m\log\alpha+n\log\beta$ arbitrarily small and this gives the result.
\\ \\
Clearly we may choose $\alpha, \beta, \gamma$ with the desired properties making the similarity dimension arbitrarily small.  In particular, the similarity dimension is the unique solution, $s$, of
\[
\alpha^s + \beta^s +  \gamma^s=1
\]
and if we choose $\alpha, \beta, \gamma$ such that $s<1$, then it follows from Corollary \ref{quasi2} (7), the above argument, and the fact that the similarity dimension is an upperbound for the upper box dimension of any self-similar set, that
\[
\dim_\text{L} F = \dim_\text{H} F = \dim_\text{B} F \leq s < 1 = \dim_\text{A} F.
\]
We give an example with $s \approx 0.901$ in Figure 3 below.

\begin{figure}[H] 
	\centering
	\includegraphics[width=80mm]{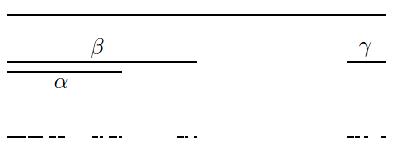}
\caption{The first level iteration and the final attractor for the self-similar set with $\alpha=2^{-\sqrt{3}}$, $\beta = 1/2$ and $\gamma=1/10$.  The tangent structure can be seen emerging around the origin.}
\end{figure}

The construction in this section has another interesting consequence.  Let $\alpha,\beta, \gamma \in (0,1)$ be chosen as before and consider the similarity maps $T_1, T_2, T_3$ on $[0,1]^2$ as follows
\[
T_1(x,y) = (\alpha x, \alpha y), \qquad T_2(x,y) = (\beta x, \beta y) + (0,1-\beta) \qquad \text{and} \qquad T_3(x) = (\gamma x, \gamma y) +(1-\gamma,0)
\]
and let $E$ be the attractor of $\{T_1, T_2, T_3\}$.  Now if $\alpha,\beta, \gamma$ are chosen such that $\alpha+\beta, \beta+\gamma, \alpha+\gamma \leq 1$ and with the similarity dimension $s<1$, then $\{T_1, T_2, T_3\}$ satisfies the open set condition and therefore by Corollary \ref{quasi2} (7) the Assouad dimension of $E$ is equal to $s$ defined above.  However, note that $F$ is the projection of $E$ onto the horizontal axis but $\dim_\text{A} F> \dim_\text{A} E$.  This shows that Assouad dimension can increase under Lipschitz maps.  This is already known, see \cite[Example A.6 2]{luk}, however, our example extends this idea in two directions as we show that the Assouad dimension can increase under Lipschitz maps on \emph{Euclidean space} and under \emph{projections}, which are a very restricted class of Lipschitz maps.

\begin{figure}[H] 
	\centering
	\includegraphics[width=70mm]{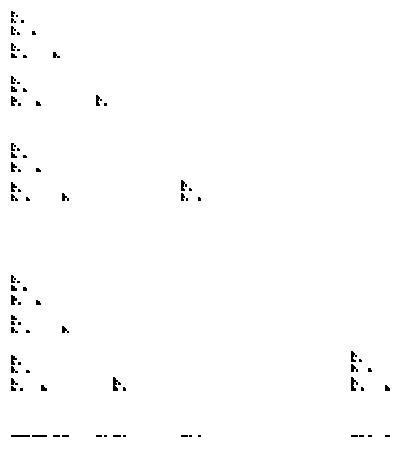}
\caption{The set $E$ and its projection $F$ for $\alpha=2^{-\sqrt{3}}$, $\beta = 1/2$ and $\gamma=1/10$.}
\end{figure}

\subsection{A self-affine carpet in the mixed class} \label{examplesB}

In this section we will give an example of a self-affine carpet in the mixed class for which $\dim_\text{L} F < \dim_\text{H} F = \dim_\text{B} F = \dim_\text{A} F$.  This is not possible in the horizontal or vertical classes by Corollary \ref{intcor1} and thus demonstrates that new phenomena can occur in the mixed class.  In particular, the dichotomy seen in Corollary \ref{intcor1} does not extend to this case.
\\ \\
For this example we will let $\{S_i\}_{i \in \mathcal{I}}$ be an IFS of affine maps corresponding to the shaded rectangles in Figure 1 below.  Here we have divided the unit square horizontally in the ratio $1/5:4/5$ and vertically into four strips each of height $1/4$.

\begin{figure}[H]
	\centering
	\includegraphics[width=120mm]{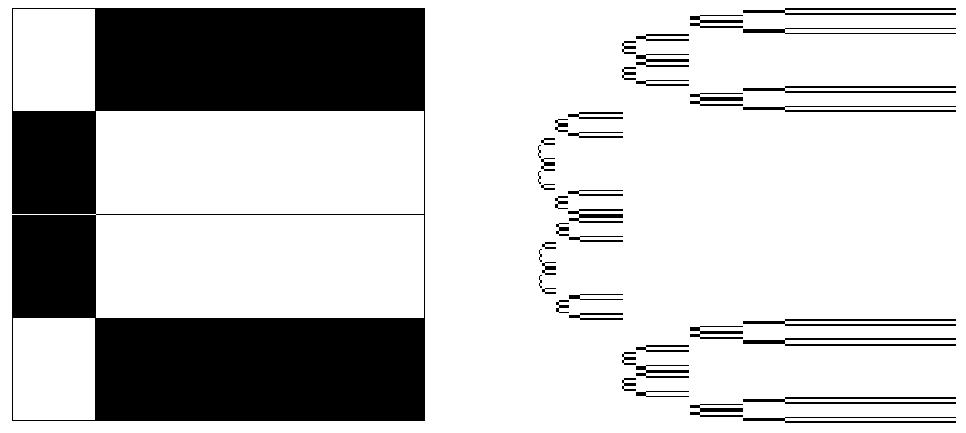}
\caption{The defining pattern for the IFS (left) and the corresponding attractor (right).}
\end{figure}

It is easy to see that
\[
\dim_\text{B} \pi_1(F) \, = \, \dim_\text{B} \pi_2(F) \, = \, 1,
\]
\[
\dim_\text{B}  \text{Slice}_{1,i}(F)  \, = \, 0.5 \qquad \text{and} \qquad \dim_\text{B}  \text{Slice}_{2,i}(F) \, = \, 0
\]
for all $i \in \mathcal{I}$ and therefore by Theorems \ref{uppermain}-\ref{lowermain}, we have $\dim_\text{L} F = 1$ and $\dim_\text{A} F = 1.5$.  Furthermore, the formulae in \cite{baranski} plus a simple calculation gives $\dim_\text{B} F =  \dim_\text{H} F = 1.5$.

\section{Open questions and discussion} \label{questions}

In this section we will briefly outline what we believe are the key questions for the future and discuss some of the interesting points raised by the results in this paper.
\\ \\
There are many natural ways to attempt to generalise our results on the Assouad and lower dimensions of self-affine sets.  Firstly, one could try to compute the dimensions of more general carpets.

\begin{ques}
What is the Assouad dimension and lower dimension of the more general self-affine carpets considered by Feng and Wang \cite{fengaffine} and Fraser \cite{me_box}?
\end{ques}

Whilst the classes of self-affine sets considered in \cite{fengaffine, me_box} are natural generalisations of the Lalley-Gatzouras and Bara\'nski classes, one notable difference is that there is no obvious analogue of \emph{approximate squares}, on which the methods used in this paper heavily rely.  In order to generalise our results one may need to `mimic' approximate squares in a delicate manor or adopt a different approach.  Secondly, one could look at higher dimensional analogues.

\begin{ques}
What is the Assouad dimension and lower dimension of the higher dimensional analogues of the self-affine sets considered here?  In particular, what are the dimensions of the Sierpi\'nski sponges (the higher dimensional analogue of the Bedford-McMullen carpet considered by, for example, Kenyon and Peres \cite{kenyonperes})?.
\end{ques}

Perhaps the most interesting direction for generalisation would be to look at arbitrary self-affine sets in a generic setting.

\begin{ques}
Can we say something about the Assouad dimension and lower dimension of self-affine sets in the generic case in the sense of Falconer \cite{affine}?  
\end{ques}

An interesting consequence of Mackay's results \cite{mackay} and Theorem \ref{uppermain} is that the Assouad dimension is not bounded above by the \emph{affinity dimension}, defined in \cite{affine}.  It is a well-known result of Falconer \cite{affine} in the dimension theory of self-affine sets that the affinity dimension is an upperbound for the upper box dimension of any self-affine set and if one randomises the translates in the defining IFS in a natural manner, then one sees that (provided the Lipschitz constants of the maps are strictly less than 1/2) the Hausdorff dimension is almost surely equal to the affinity dimension.  Are the Assouad and lower dimensions almost surely equal?  If they are, then this almost sure value must indeed be the affinity dimension.  If they are not almost surely equal, then are they at least almost surely equal to two different constants?
\\ \\
In the study of fractals one is often concerned with measures supported on sets rather than sets themselves.  Although their definitions depend only on the structure of the set, the Assouad and lower dimensions have a fascinating link with certain classes of measures.  Luukkainen and Saksman \cite{luksak} (see also \cite{konyagin}) proved that the Assouad dimension of a compact metric space $X$ is the infimum of $s \geq 0$ such that there exists a locally finite measure $\mu$ on $X$ and a constant $c_s > 0$ such that for any $0 < \rho < 1$, $x \in X$ and $r > 0$
\begin{equation} \label{uppermeasure1}
\mu(B(x,r)) \ \leq \ c_s \rho^{-s} \mu(B(x,\rho r)).
\end{equation}
Dually, Bylund and Gudayol \cite{bylund} proved that the lower dimension of a compact metric space $X$ is the supremum of $s \geq 0$ such that there exists a locally finite measure $\mu$ on $X$ and a constant $d_s > 0$ such that for any $0 < \rho < 1$, $x \in X$ and $r > 0$
\begin{equation} \label{lowermeasure1}
\mu(B(x,r)) \ \geq \ d_s \rho^{-s} \mu(B(x,\rho r)).
\end{equation}
As such, our results give the existence of measures supported on self-affine carpets with useful scaling properties.  In particular, if $F$ is a self-affine carpet, then for each $s>\dim_\text{A} F$ there exists a measure supported on $F$ satisfying (\ref{uppermeasure1}) and  for each $s<\dim_\text{L} F$ there exists a measure supported on $F$ satisfying (\ref{lowermeasure1}).  It is natural to ask if `sharp' measures exist.

\begin{ques}
Let $F$ be a self-affine carpet.  Does there exist a measure supported on $F$ satisfying (\ref{uppermeasure1}) for $s=\dim_\text{\emph{A}} F$ and a measure supported on $F$ satisfying (\ref{lowermeasure1}) for $s=\dim_\text{\emph{L}} F$?
\end{ques}

As mentioned above, it is interesting to examine the relationship between the Assouad and lower dimensions and the other dimensions discussed here.  In particular, for a given class of sets one can ask \emph{what relationships are possible between the dimensions?}  For example, for Ahlfors regular sets all the dimension are necessarily equal.  The following table summarises the possible relationships between the Assouad and lower dimensions and the box dimension for the classes of sets we have been most interested in.

\begin{center}
    \begin{tabular}{ | l | c| c| c|}
    \hline
    Configuration & horizontal/vertical class & mixed class & self-similar class\\ \hline 
  $\dim_\text{L} F  \ = \  \dim_\text{B} F \ = \  \dim_\text{A} F$ & possible & possible & possible \\ \hline 
  $\dim_\text{L} F  \ = \  \dim_\text{B} F \ < \  \dim_\text{A} F$ & not possible & not possible& possible \\   \hline
    $\dim_\text{L} F  \ < \  \dim_\text{B} F \ = \  \dim_\text{A} F$ & not possible & possible & not possible \\ \hline  
 $\dim_\text{L} F  \ < \  \dim_\text{B} F \ < \  \dim_\text{A} F$ & possible & possible & not possible \\   \hline
    \end{tabular}
\end{center}

The information presented in this table can be gleaned from Corollary \ref{quasi2}, Corollary \ref{intcor1}, Corollary \ref{intcor2} and the examples in Sections \ref{selfsimexample} and \ref{examplesB}.  Interestingly, the configuration $\dim_\text{L} F  \ < \  \dim_\text{B} F \ = \  \dim_\text{A} F$ is possible for self-affine carpets, but not for self-similar sets (even with overlaps) and the configuration $\dim_\text{L} F  \ = \  \dim_\text{B} F \ < \  \dim_\text{A} F$ is not possible for self-affine carpets, but is possible for self-similar sets with overlaps.  Roughly speaking, the reason for this is that the non-uniform scaling present in self-affine carpets allows one to `spread' the set out making certain places easier to cover and thus making the lower dimension drop and one can use overlaps to `pile' the set up making certain places harder to cover and thus raising the Assouad dimension.  It would be interesting to add Hausdorff dimension to the above analysis, but there are some configurations for which we have been unable to determine if they are possible or not.

\begin{ques}

Are any of the entries marked with a question mark in the following table possible in the relevant class of sets?  The rest of the entries may be gleaned from Corollary \ref{quasi2}, Corollary \ref{intcor1}, Corollary \ref{intcor2} and the examples in Sections \ref{selfsimexample} and \ref{examplesB}.
\vspace{-4mm}
\emph{\begin{center}
    \begin{tabular}{ | l | c| c| c|}
    \hline
   Configuration & horizontal/vertical class & mixed class & self-similar class\\ \hline 
 $\dim_\text{L} F  \ = \  \dim_\text{H} F \ = \  \dim_\text{B} F \ = \  \dim_\text{A} F$ & possible & possible & possible \\ \hline 
 $\dim_\text{L} F  \ = \  \dim_\text{H} F \ = \ \dim_\text{B} F \ < \  \dim_\text{A} F$ & not possible & not possible& possible \\   \hline
 $\dim_\text{L} F  \ = \  \dim_\text{H} F \ < \  \dim_\text{B} F \ = \  \dim_\text{A} F$ & not possible & ? & not possible \\ \hline 
 $\dim_\text{L} F  \ = \  \dim_\text{H} F \ < \ \dim_\text{B} F \ < \  \dim_\text{A} F$ & not possible & ? & not possible \\   \hline
 $\dim_\text{L} F  \ < \  \dim_\text{H} F \ = \   \dim_\text{B} F \ = \  \dim_\text{A} F$ & not possible & possible & not possible \\ \hline
 $\dim_\text{L} F  \ < \ \dim_\text{H} F \ = \  \dim_\text{B} F \ < \  \dim_\text{A} F$ & not possible & ? & not possible \\   \hline
 $\dim_\text{L} F  \ < \  \dim_\text{H} F \ < \   \dim_\text{B} F \ = \  \dim_\text{A} F$ & not possible & ? & not possible \\ \hline
 $\dim_\text{L} F  \ < \ \dim_\text{H} F \ < \  \dim_\text{B} F \ < \  \dim_\text{A} F$ & possible & possible & not possible \\   \hline
    \end{tabular}
\end{center}}

\end{ques}

Finally, it is natural to wonder what precise Baire class the lower dimension belongs to, especially given that we have proved that the Assouad dimension is precisely Baire 2.  We have proved that it is no worse than Baire 3, see Theorem \ref{meas2}, and not Baire 1, so it remains to decide whether the lower dimension is Baire 2, and thus has the same level of complexity as the Assouad dimension and Hausdorff dimension, or is not Baire 2, and thus is more complex than the Assouad dimension and Hausdorff dimension.

\begin{ques}
Is the lower dimension always of Baire class 2?
\end{ques}

\section{Proofs of basic properties} \label{proofs1}

Throughout this section, $X$ and $Y$ will be metric spaces.  

\subsection{Proof of Theorem \ref{lowerassproduct}: products} \label{basicproductproof}

We will prove that $\dim_\text{L} X \, + \, \dim_\text{L} Y \leq \dim_\text{L} (X \times Y)  \leq  \dim_\text{L} X \, + \, \dim_\text{A} Y$.  The proof of the analogous formula for the Assouad dimension of $X \times Y$ is similar and omitted.  Write $M_r(F)$ to denote the maximum cardinality of an $r$-separated subset of a set $F$ where an $r$-separated set is a set where the the distance between any two pairs of points is strictly greater than $r$.  Observe that for all $(x,y) \in X \times Y$ and all $0<r<R$, we have
\begin{equation} \label{Nproduct}
N_r\Big( B\big((x,y),R\big) \cap (X\times Y) \Big) \ \leq \  N_r\big( B(x,R) \cap X \big) \,N_r\big( B(y,R) \cap Y \big)
\end{equation}
and
\begin{equation} \label{Mproduct}
M_r\Big( B\big((x,y),R\big) \cap (X\times Y) \Big) \ \geq \  M_r\big( B(x,R) \cap X \big) \, M_r\big( B(y,R) \cap Y \big).
\end{equation}
The first inequality follows since if $\{U_i\}_i$, $\{V_i\}_i$ are arbitrary $r$-covers of $B(x,R) \cap X$ and  $B(y,R) \cap Y$ respectively, then $\{U_i \times V_j\}_{i,j}$ is an $r$-cover of $B\big((x,y),R\big) \cap (X\times Y)$ and the second inequality follows since if $\{u_i\}_i$, $\{v_i\}_i$ are arbitrary $r$-separated subsets of $B(x,R) \cap X$ and  $B(y,R) \cap Y$ respectively, then $\{(u_i, v_j)\}_{i,j}$ is an $r$-separated subset of $B\big((x,y),R\big) \cap (X\times Y)$.
\\ \\
\emph{Proof of lower bound.}  Let $s<\dim_\text{L} X$ and $t<\dim_\text{L} Y$.  It follows that there exists $C_X, C_Y, \rho_X, \rho_Y$ such that for all $0<r<R<\rho_X$ and $x \in X$ we have
\[
M_r\big( B(x,R) \cap X \big) \ \geq \ C_X \, \bigg(\frac{R}{r} \bigg)^s
\]
and for all $0<r<R<\rho_Y$ and $y \in Y$ we have
\[
M_r\big( B(y,R) \cap Y \big)\  \geq \ C_Y \, \bigg(\frac{R}{r} \bigg)^t.
\]
It now follows from (\ref{Mproduct}) that, for all $0<r<R<\min\{\rho_X, \rho_Y\}$ and all $(x, y) \in X \times Y$, we have
\[
M_r\Big( B\big((x,y),R\big) \cap (X\times Y) \Big) \ \geq \ C_X \, C_Y \, \bigg(\frac{R}{r} \bigg)^{s+t}
\]
which implies that $\dim_\text{L} (X \times Y)  \geq s+t$ which proves the desired lower bound letting $s \nearrow \dim_\text{L} X$ and $t \nearrow \dim_\text{L} Y$. \hfill \qed
\\ \\
\emph{Proof of upper bound.}  Let $C, \rho>0$ and let $s>\dim_\text{L} X$ and $t>\dim_\text{A} Y$.   It follows that there exists $C_Y, \rho_Y$ such that for all $0<r<R<\rho_Y$ and $y \in Y$ we have
\[
N_r\big( B(y,R) \cap Y \big)\  \leq \ C_Y \, \bigg(\frac{R}{r} \bigg)^t
\]
and there exists $0<r_X<R_X< \min\{\rho, \rho_Y\}$ and $x_X \in X$ such that
\[
N_{r_X}\big( B(x,R_X) \cap X \big) \  < \ \frac{C}{C_Y} \, \bigg(\frac{R_X}{r_X} \bigg)^s.
\]
It now follows from (\ref{Nproduct}) that for any $y \in Y$ that
\[
N_{r_X}\Big( B\big((x_X,y),R_X\big) \cap (X\times Y) \Big) \ < \ \frac{C}{C_Y} \, \bigg(\frac{R_X}{r_X} \bigg)^s \, C_Y \, \bigg(\frac{R_X}{r_X} \bigg)^t \ =  \ C \, \bigg(\frac{R_X}{r_X} \bigg)^{s+t}
\]
which implies that $\dim_\text{L} (X \times Y)  \leq s+t$ which proves the desired upper bound letting $s \searrow \dim_\text{L} X$ and $t \searrow \dim_\text{A} Y$. \hfill \qed
\\ \\
Finally we will prove that if $Y=X$, then we can obtain a sharp result for the lower dimension of the product.  In fact, we will prove that $\dim_\text{L} (X^n) \ = \ n\, \dim_\text{L} X$.  The fact that $\dim_\text{L} (X^n) \ \geq \ n\, \dim_\text{L} X$ follows from the above so we will now prove the other direction.

\begin{proof}
Let $C, \rho > 0$ and $s>\dim_\text{L} X$.  It follows that there exists $0<r<R< \rho$ and $x \in X$ such that
\[
N_{r}\big( B(x,R) \cap X \big) \  < \ \sqrt[n]{C} \, \bigg(\frac{R}{r} \bigg)^s.
\]
and by repeatedly applying (\ref{Nproduct}) we obtain
\[
N_{r}\Big( B\big((\underbrace{x, \dots, x}_{n \text{ times}}),R\big) \cap (X^n) \Big) \ < \ ( \sqrt[n]{C} )^n \, \bigg(\frac{R}{r} \bigg)^{ns} \ = \ C\, \bigg(\frac{R}{r} \bigg)^{ns}
\]
which implies that $\dim_\text{L} (X^n)  \leq ns$ which proves the desired upper bound letting $s \searrow \dim_\text{L} X$.
\end{proof}

\subsection{Proof of Theorem \ref{lowerassunion}: unions} \label{basicunionproof}

Let $E,F \subseteq X$.  The inequality $\min \{ \dim_\text{L} E,  \ \dim_\text{L}  F\}  \leq  \dim_\text{L} (E \cup F)$ is trivial since if $x \in E$, then we use the estimate
\[
N_r\big(B(x,R) \cap (E \cup F) \big) \ \geq  \ N_r\big(B(x,R) \cap E \big)
\]
to obtain the desired scaling and if $x \in F$, then we use the estimate
\[
N_r\big(B(x,R) \cap (E \cup F) \big) \ \geq  \ N_r\big(B(x,R) \cap F \big).
\]
We will now prove the other direction.  

\begin{proof}
Fix $C, \rho>0$ and let $t>s>\max \{ \dim_\text{L} E,  \ \dim_\text{A}  F\}$.  Since $s> \dim_\text{A} F$, there exists $C_F, \, \rho_F>0$ such that for all $0<r<R\leq \rho_F$ and all $x \in X$, we have
\begin{equation} \label{upperscaling}
N_r\big( B(x,R) \cap F \big) \ \leq \ C_F \bigg(\frac{R}{r}\bigg)^s \ = \ C_F \bigg(\frac{R}{r}\bigg)^{s-t}  \bigg(\frac{R}{r}\bigg)^t.
\end{equation}
Technically speaking the definition of Assouad dimension only gives that the estimate (\ref{upperscaling}) holds for $x \in F$, however, we will need it to hold for all $x \in X$.  To see why we can assume this, note that the intersection of $F$ with any ball centered in $X$ is either empty or contained in a ball centered in $F$ with double the radius.  Also, since $t> \dim_\text{L} E$, there exists $0<r<R<\min\{\rho, \rho_F\}$ and $x \in E$ such that
\begin{equation} \label{lowerscaling}
N_r\big( B(x,R) \cap E \big) \ < \ \min \bigg\{C/2, \, \Big(\frac{2C_F}{C}\Big)^{t/(s-t)}\bigg\} \,  \bigg(\frac{R}{r}\bigg)^t.
\end{equation}
Observe that
\[
1 \ \leq \ N_r\big( B(x,R) \cap E \big) \ \leq \   \Big(\frac{2C_F}{C}\Big)^{t/(s-t)}  \,  \bigg(\frac{R}{r}\bigg)^t
\]
and so
\begin{equation} \label{thetaest}
\bigg(\frac{R}{r}\bigg)^{s-t} \  \leq \  \frac{C}{2C_F}.
\end{equation}
By (\ref{upperscaling}) and (\ref{lowerscaling}), there exists $0<r<R\leq \rho$ and $x \in E \subseteq  E \cup F$ such that
\begin{eqnarray*}
N_r\big( B(x,R) \cap (E\cup F) \big) & \leq & N_r\big( B(x,R) \cap E \big) \ + \ N_r\big( B(x,R) \cap  F \big) \\ \\
&<& \min \bigg\{C/2, \, \Big(\frac{2C_F}{C}\Big)^{t/(s-t)}\bigg\}  \,  \bigg(\frac{R}{r}\bigg)^t \ + \ C_F \bigg(\frac{R}{r}\bigg)^{s-t}  \bigg(\frac{R}{r}\bigg)^t \\ \\
&\leq& C/2 \,  \bigg(\frac{R}{r}\bigg)^t \ + \ C_F \, \frac{C}{2C_F} \bigg(\frac{R}{r}\bigg)^t \qquad \qquad \text{by (\ref{thetaest})}\\ \\
&=& C \,  \bigg(\frac{R}{r}\bigg)^t
\end{eqnarray*}
which proves that $\dim_\text{L} (E \cup F) \leq \max \{ \dim_\text{L} E,  \ \dim_\text{A}  F\}$.
\end{proof}
Finally, to complete the proof of Theorem \ref{lowerassunion} assume that $E$ and $F$ are such that $\inf_{x \in E, \, y \in F} d(x,y) = \eta >0$.  It follows that any ball centered in $E \cup F$ with radius $R<\eta$ can only intersect one of $E$ and $F$ from which it easily follows that $\dim_\text{L} (E \cup F) = \min \{ \dim_\text{L} E,  \ \dim_\text{L}  F\}$.

\subsection{Proof of Theorem \ref{closurestable2}: closures} \label{basicclosureproof}

Let $F \subseteq X$.  Since lower dimension is not monotone, we have to prove that $\dim_\text{L}  F \leq \dim_\text{L}  \overline{F}$ and $\dim_\text{L}  F \geq \dim_\text{L}  \overline{F}$.  We will prove $\dim_\text{L} F \leq  \dim_\text{L} \overline{F}$ and argue that the other direction follows by a similar argument. 

\begin{proof}
Let $s> \dim_\text{L} \overline{F}$ and fix $C, \rho>0$.  It follows that there exists $\overline{x} \in \overline{F}$ and $r, R>0$ with $0<r<R<\rho$ such that
\begin{equation} \label{closexp1}
N_r\big( B(\overline{x}, R) \cap \overline{F} \big) \ < \ C \, 2^{-s} \,   \bigg(\frac{R}{r} \bigg)^s.
\end{equation}
Let $\varepsilon \in (0, R/2)$ and choose $x \in F \cap B(\overline{x}, \varepsilon)$.  It follows that
\begin{equation} \label{closexp}
B(x, R-\varepsilon) \cap F \ \subseteq \  B(\overline{x}, R) \cap \overline{F},
\end{equation}
and hence
\begin{eqnarray*}
N_r\big( B(x, R-\varepsilon) \cap F  \big)  \ \leq \  N_r\big( B(\overline{x}, R) \cap \overline{F} \, \big) &<& C \,2^{-s}  \,  \bigg(\frac{R}{r} \bigg)^s \qquad  \qquad \text{by (\ref{closexp1})}\\ \\
&\leq& C  \,2^{-s} \, \bigg(\frac{R}{R-\varepsilon} \bigg)^s \, \bigg(\frac{R-\varepsilon}{r} \bigg)^s \\ \\
&\leq& C  \, \bigg(\frac{R-\varepsilon}{r} \bigg)^s,
\end{eqnarray*}
which proves that $\dim_\text{L} F \leq  s$ and letting $s \searrow \dim_\text{L} \overline{F}$ gives the desired estimate.  The proof of the opposite inequality is similar and we only sketch it.  In this case we first choose $x \in F$, then $\overline{x} \in \overline{X} \cap B(x, \varepsilon)$ and obtain
\[
B(\overline{x}, R-\varepsilon) \cap \overline{F} \ \subseteq \  \overline{B(x, R) \cap F}.
\]
Now observe that if $\{U_i\}_i$ is a cover of $B(x, R) \cap F$ by closed balls, then $\{U_i\}_i$ is also a cover of $B(\overline{x}, R-\varepsilon) \cap \overline{F}$ and so using closed balls in the definition of $N_r$ we can complete the proof as above.
\end{proof}

\subsection{Proof of Theorem \ref{bi-Lipschitz}: bi-Lipschitz invariance} \label{basicbi-Lipschitzproof}

Suppose $\phi : X \to Y$ is an onto bi-Lipschitz mapping with Lipschitz constants $a,b >0$ such that
\[
a \lvert x-y\rvert \leq \lvert \phi(x) - \phi(y) \rvert \leq b \lvert x-y \rvert
\]
for $x,y \in X$.  We will prove that $\dim_\text{L} X \geq  \dim_\text{L} Y$ and observe that the other direction follows by the same argument using $\phi^{-1}$.
\begin{proof}
Let $s<\dim_\text{L} Y$.  It follows that there exists $C, \rho>0$ such that for all $0<r$, $0<R<\rho$ and $y \in Y$ we have
\[
N_r\big( B(y,R) \cap Y \big) \ \geq \ C \, \bigg(\frac{R}{r} \bigg)^s.
\]
Since $B(\phi(x), ar) \subseteq \phi\big(B(x, r)\big)$ for all $x \in X$ and $r>0$ and $\text{diam}\big(\phi(U)\big) \leq b \, \text{diam}(U)$ for all sets $U \subseteq X$, it follows that
\[
N_r\big( B(x,R) \cap X \big) \ \geq \ N_{br}\big( B(\phi(x),aR) \cap Y \big) \ \geq \  C \, \bigg(\frac{aR}{br} \bigg)^s \ = \ C \, (a/b)^s \, \bigg(\frac{R}{r} \bigg)^s 
\]
provided $R < \rho/a$ which shows that $\dim_\text{L} X \geq  s$ and letting $s \nearrow \dim_\text{L} Y$ proves the result.
\end{proof}
Note that we use both Lipschitz constants in this proof which is consistent with the fact that maps which are only Lipschitz on one side do not necessarily preserve the dimension in either direction.

\subsection{Proof of Theorems \ref{meas1} and \ref{meas2}: measurability of the Assouad and lower dimensions} \label{basicmeasureproof1}

Throughout this section $(X,d)$ will be a compact metric space, $\overline{B}(x,R)$ will denote the \emph{closed} ball centered at $x \in X$ with radius $R>0$ and $B^0(x,R)$ will denote the \emph{open} ball centered at $x \in X$ with radius $R>0$.  For $x \in X$ and $R>0$ define a map $\overline{\beta}_{x,R}: \mathcal{K}(X)  \to \mathcal{K}_0(X)$ by
\[
\overline{\beta}_{x,R}(F) = \overline{B}(x,R) \cap F
\]
and a map $\beta^0_{x,R}: \mathcal{K}(X)  \to \mathcal{P}(X)$ by
\[
\beta^0_{x,R}(F) = B^0(x,R) \cap F.
\]
Also, $N_r(F)$ will denote the smallest number of \emph{open} sets required for an $r$-cover of $F \subseteq X$ and $M_r(F)$ will denote the maximum number of \emph{closed} sets in an $r$-packing of $F\subseteq X$, where an $r$-packing of $F$ is a pairwise disjoint collection of closed balls centered in $F$ of radius $r$.

\begin{lma} \label{uppersemicontinuity2}
Let $x \in X$ and $R,r>0$.  The map $N_r \circ \overline{\beta}_{x,R} : \mathcal{K}(X) \to \mathbb{R}$ is upper semicontinuous.
\end{lma}

\begin{proof}
It was proved in \cite{mattilamauldin} that the function $N_r$ is upper semicontinuous, however, the function $\overline{\beta}_{x,R}$ is clearly not continuous and so we cannot apply their result directly.  Nevertheless, the proof is similar and straightforward.  Let $F \in \mathcal{K}(X) $ and let $\{U_i\}$ be an open $r$-cover of $B(x,R) \cap F$.  Observe that
\[
\eta \ =  \ \frac{1}{2} \,  \inf_{\substack{y \in F \\ \\
 z \in B(x,R) \setminus ( \cup_i U_i )} } d(y,z)
\]
is strictly positive since $F$ and $B(x,R) \setminus ( \cup_i U_i )$ are compact and non-intersecting.  It follows that if $E \in \mathcal{K}(X) $ is such that $d_\mathcal{H}(E,F)< \eta$, then the sets $\{U_i\}$ form an open $r$-cover of $B(x,R) \cap E$, from which it follows that sets of the form
\[
\{F \in \mathcal{K}(X) :  \big(N_r \circ \overline{\beta}_{x,R}\big)( F)< t \}  \qquad \qquad ( t \in \mathbb{R})
\]
are open which gives upper semicontinuity.
\end{proof}

\begin{lma} \label{lowersemicontinuity}
Let $x \in X$ and $R,r>0$.  The map $M_r \circ \beta^0_{x,R} : \mathcal{K}(X) \to \mathbb{R}$ is lower semicontinuous.
\end{lma}

\begin{proof}
It was proved in \cite{mattilamauldin} that the function $M_r$ is lower semicontinuous, however, as before we cannot apply this result directly.  Let $F \in \mathcal{K}(X) $ and let $\{\overline{B}(y_i,r)\}_{i \in \Lambda}$ be an $r$-packing of $B^0(x,R) \cap F$ by closed balls with centres $\{y_i\}$ in $B^0(x,R) \cap F$.  Observe that
\[
\eta \  =  \ \frac{1}{2} \, \min\Big\{ \min_{\substack{ i,j \in \Lambda: \\ i \neq j}} d(y_i,y_j)-2r, \ \min_{\substack{i \in \Lambda\\  z \notin B^0(x,R)}} d(y_i,z)  \Big\}
\]
is strictly positive.  It follows that if $E \in \mathcal{K}(X) $ is such that $d_\mathcal{H}(E,F)< \eta$, then we can find an $r$-packing of $B(x,R) \cap E$ by $\lvert \Lambda \rvert$ closed balls, from which it follows that sets of the form
\[
\{F \in \mathcal{K}(X) :  \big(N_r \circ \overline{\beta}_{x,R}\big)( F)> t \}  \qquad \qquad ( t \in \mathbb{R})
\]
are open which gives lower semicontinuity.
\end{proof}

The following lemma gives a useful equivalent definition of Assouad dimension.

\begin{lma} \label{equivdef}
For a subset $F$ of a compact metric space $X$ we have
\begin{eqnarray*}
\dim_\text{\emph{A}} F \ = \  \inf \Bigg\{ &\alpha& : \text{     there exists constants $C, \, \rho>0$ such that,} \\
&\,& \text{ for all $0<r<R\leq \rho$, we have $\ \sup_{x \in X_0} \, N_r\big( B(x,R) \cap F\big) \ \leq \ C \bigg(\frac{R}{r}\bigg)^\alpha$ } \Bigg\}
\end{eqnarray*}
for any dense subset $X_0$ of $X$.
\end{lma}

\begin{proof}
Let $X_0$ be a dense subset of a compact metric space $X$ and let $\dim^{X_0}_\text{A} F$ be the definition given on the right hand side of the equation above (which appears to depend on $X_0$).  The fact that $\dim^{X_0}_\text{A} F \geq \dim_\text{A} F$ follows immediately since any ball centered in $F$ with radius $r$ contains, and is contained in, a ball centered in $X_0$ with radius arbitrarily close to $r$.  The opposite inequality follows from the fact that the intersection of $F$ with any ball centered in $X_0$ is contained in a ball with double the radius, centered in $F$.
\end{proof}

We remark here that there does not exist a similar alternative definition for lower dimension.  The reason for this is that (as long as $F$ is not dense) one can place a ball, $B(x,R)$, completely outside the set $F$, causing $\inf_{x \in X_0} \, N_r\big( B(x,R) \cap F\big)$ to be equal to zero for sufficiently small $R$.

\begin{lma} \label{borelset1}
For all $t \in \mathbb{R}$, the set
\[
\{F \in \mathcal{K}(X) :  \dim_\text{\emph{A}} F< t \}
\]
is $\mathcal{G}_{\delta \sigma}$.
\end{lma}

\begin{proof}

Let $X_0$ be a countable dense subset of $X$, which exists because $X$ is compact and thus separable, and let $t \in \mathbb{R}$.  Using Lemma \ref{equivdef}, we have
\begin{eqnarray*}
\{F \in \mathcal{K}(X) :  \dim_\text{A} F<t \} \hspace{-30mm} \\ \\
&=& \Bigg\{F \in \mathcal{K}(X) : \text{for some $n \in \mathbb{N}$, there exists } C, \rho>0 \text{ such that for all } x \in X_0 \\ \\
&\,& \qquad \qquad \qquad  \text{ and all } 0<r<R< \rho, \text{ we have } \quad N_r \Big( \overline{B}\big( x, R \big) \cap F \Big)  \ < \ C \, \bigg( \frac{R}{r} \bigg)^{t-1/n} \Bigg\} \\ \\
&=&  \bigcup_{n \in \mathbb{N}} \  \bigcup_{C \in \mathbb{N}} \  \bigcup_{\rho \in \mathbb{Q}^+} \ \bigcap_{x \in X_0} \  \bigcap_{R \in \mathbb{Q} \cap (0,\rho)} \ \bigcap_{r \in \mathbb{Q} \cap (0,R)} \ \Bigg\{ F \in \mathcal{K}(X) :  N_r \big( \overline{\beta}_{x,R}(F)  \big)  \ < \ C \, \bigg( \frac{R}{r} \bigg)^{t-1/n} \Bigg\} \\ \\
&=&  \bigcup_{n \in \mathbb{N}} \  \bigcup_{C \in \mathbb{N}} \  \bigcup_{\rho \in \mathbb{Q}^+} \ \bigcap_{x \in X_0} \  \bigcap_{R \in \mathbb{Q} \cap (0,\rho)} \ \bigcap_{r \in \mathbb{Q} \cap (0,R)} \ \big(N_r \circ \overline{\beta}_{x,R}\big)^{-1}\,  \Big( \big( -\infty,  \, C \, (R/r)^{t-1/n} \big) \Big).
\end{eqnarray*}
The set $\big(N_r \circ \overline{\beta}_{x,R}\big)^{-1}\,  \Big( \big( -\infty,  \, C \, (R/r)^{t-1/n} \big) \Big)$ is open by the upper semicontinuity of $N_r \circ \overline{\beta}_{x,R}$, see Lemma \ref{uppersemicontinuity2}.  It follows that $\{F \in \mathcal{K}(X) :  \dim_\text{A} F < t\}$ is a $\mathcal{G}_{\delta \sigma}$ subset of $\mathcal{K}(X)$.
\end{proof}

\begin{lma} \label{borelset11}
For all $t \in \mathbb{R}$, the set
\[
\{F \in \mathcal{K}(X) :  \dim_\text{\emph{A}} F > t \}
\]
is $\mathcal{G}_{\delta \sigma}$.
\end{lma}

\begin{proof}

Let $t \in \mathbb{R}$.  We have
\begin{eqnarray*}
\{F \in \mathcal{K}(X) :  \dim_\text{A} F>t \} \hspace{-30mm} \\ \\
&=& \Bigg\{F \in \mathcal{K}(X) : \text{there exists $n \in \mathbb{N}$ such that for all } C, \rho>0 \text{ there exists } x \in X \\ \\
&\,& \qquad \qquad \qquad  \text{ and } 0<r<R< \rho, \text{ such that } \quad M_r \Big( B^0\big( x, R \big) \cap F \Big)  \ > \ C \, \bigg( \frac{R}{r} \bigg)^{t+1/n} \Bigg\} \\ \\
&=&  \bigcup_{n \in \mathbb{N}} \  \bigcap_{C \in \mathbb{N}} \  \bigcap_{\rho \in \mathbb{Q}^+} \ \bigcup_{x \in X} \  \bigcup_{R \in \mathbb{Q} \cap (0,\rho)} \ \bigcup_{r \in \mathbb{Q} \cap (0,R)} \ \Bigg\{ F \in \mathcal{K}(X) :  M_r \big( \beta^0_{x,R}(F)  \big)  \ > \ C \, \bigg( \frac{R}{r} \bigg)^{t+1/n} \Bigg\} \\ \\
&=&  \bigcup_{n \in \mathbb{N}} \  \bigcap_{C \in \mathbb{N}} \  \bigcap_{\rho \in \mathbb{Q}^+} \ \bigcup_{x \in X} \  \bigcup_{R \in \mathbb{Q} \cap (0,\rho)} \ \bigcup_{r \in \mathbb{Q} \cap (0,R)} \ \big(M_r \circ \beta^0_{x,R}\big)^{-1}\,  \Big( \big(  C \, (R/r)^{t+1/n} , \infty \big) \Big).
\end{eqnarray*}
The set $\big(M_r \circ \beta^0_{x,R}\big)^{-1}\,  \Big( \big(  C \, (R/r)^{t+1/n}, \infty \big) \Big)$ is open by the lower semicontinuity of $M_r \circ \beta^0_{x,R}$, see Lemma \ref{lowersemicontinuity}.  It follows that $\{F \in \mathcal{K}(X) :  \dim_\text{A} F > t\}$ is a $\mathcal{G}_{\delta \sigma}$ subset of $\mathcal{K}(X)$.
\end{proof}

Theorem \ref{meas1} now follows easily.
\begin{proof}
To show that $\Delta_\text{A}$ is Baire 2, it suffices to show that all open sets of the form $(t, u)$, for $t,u \in \mathbb{R}$ with $t<u$, are pulled back to $\mathcal{G}_{\delta \sigma}$ sets, see \cite[Theorem 24.3]{kechris}.  For such $t,u$, we have
\begin{eqnarray*}
\Delta_\text{A}^{-1} \big( (t,u) \big) &=& \{F \in \mathcal{K}(X) :  \dim_\text{A} F > t \} \ \bigcap \  \{F \in \mathcal{K}(X) :  \dim_\text{A} F < u \}
\end{eqnarray*}
and it follows from Lemmas \ref{borelset1} and \ref{borelset11} that this set is $\mathcal{G}_{\delta \sigma}$.
\end{proof}

We will now turn to the proof of Theorem \ref{meas2}.  One important difference is that we do not have an analogue of Lemma \ref{equivdef} for lower dimension.  To get round this problem, instead of dealing with the continuity properties of the functions $N_r \circ \overline{\beta}_{x,R}$ and $M_r \circ \beta^0_{x,R}$ we must deal with the more complicated function $\Phi_{r,R}$ defined as follows. For $R,r>0$, let $\Phi_{r,R}:\mathcal{K}(X) \to \mathbb{R}$ be defined by
\[
\Phi_{r,R}(F) = \inf_{x \in F} N_r \Big( \overline{B}\big( x, R \big) \cap F \Big).
\]

\begin{lma} \label{uppersemicontinuity3}
Let $R,r>0$.  The map $\Phi_{r,R}$ is upper semicontinuous.
\end{lma}

\begin{proof}
This follows easily from Lemma \ref{uppersemicontinuity2}, observing that if $F \in \mathcal{K}(X)$ is such that $\Phi_{r,R} (F) < t$ for some $t \in \mathbb{R}$, then there must exist a point $x \in F$ such that $N_r \circ \overline{\beta}_{x,R}(F) <t$.  We can then simply apply the upper semincontinuity of $N_r \circ \overline{\beta}_{x,R}$ to complete the proof.
\end{proof}

\begin{lma} \label{borelset2}
For all $t \in \mathbb{R}$, the set
\[
\{F \in \mathcal{K}(X) : \dim_\text{\emph{L}} F<t \}
\]
is $\mathcal{G}_{\delta \sigma}$.
\end{lma}

\begin{proof}
Let $t \in \mathbb{R}$.  We have
\begin{eqnarray*}
\{F \in \mathcal{K}(X) :  \dim_\text{L} F<t \} \hspace{-30mm} \\ \\
&=& \Bigg\{F \in \mathcal{K}(X) : \text{ there exists $n \in \mathbb{N}$ such that for all } C, \rho>0 
\text{ there exists } 0<r<R< \rho \\ \\
&\,& \qquad \qquad \qquad  \text{ and } x \in F, \text{ such that } \quad N_r \Big( \overline{B} \big( x, R \big) \cap F \Big)  \ < \ C \, \bigg( \frac{R}{r} \bigg)^{t-1/n} \Bigg\} \\ \\
&=& \bigcup_{n \in \mathbb{N}} \ \bigcap_{C \in \mathbb{N}} \  \bigcap_{\rho \in \mathbb{Q}^+} \   \bigcup_{R \in \mathbb{Q} \cap (0,\rho)} \ \bigcup_{r \in \mathbb{Q} \cap (0,R)} \ \Bigg\{ F \in \mathcal{K}(X) :  \inf_{x \in F} N_r \Big( \overline{B} \big( x, R \big) \cap F \Big)  \ < \ C \, \bigg( \frac{R}{r} \bigg)^{t-1/n} \Bigg\} \\ \\
&=& \bigcup_{n \in \mathbb{N}} \ \bigcap_{C \in \mathbb{N}} \  \bigcap_{\rho \in \mathbb{Q}^+} \   \bigcup_{R \in \mathbb{Q} \cap (0,\rho)} \ \bigcup_{r \in \mathbb{Q} \cap (0,R)} \ \Phi_{r,R}^{-1} \ \Big( \big( - \infty, \, C \, (R/r)^{t-1/n} \big) \Big).
\end{eqnarray*}
The set $\Phi_{r,R}^{-1} \ \Big( \big( - \infty, \, C \, (R/r)^{t-1/n} \big) \Big)$ is open by the upper semicontinuity of $\Phi_{r,R}$, see Lemma \ref{uppersemicontinuity3}.  It follows that $\{F \in \mathcal{K}(X) : \dim_\text{L} F < t  \}$ is a $\mathcal{G}_{\delta \sigma}$ subset of $\mathcal{K}(X)$.
\end{proof}

Theorem \ref{meas2} now follows easily.

\begin{proof}
To show that $\Delta_\text{L}$ is Baire 3, it suffices to show that all open sets of the form $(t, u)$, for $t,u \in \mathbb{R}$ with $t<u$, are pulled back to $\mathcal{F}_{\sigma \delta \sigma}$ sets, see \cite[Theorem 24.3]{kechris}.  For such $t,u$ and writing $\mathcal{Y}^c = \mathcal{K}(X) \setminus \mathcal{Y}$ for the complement of a set $\mathcal{Y} \subseteq \mathcal{K}$ we have
\begin{eqnarray*}
\Delta_\text{L}^{-1} \big( (t,u) \big) &=& \{F \in \mathcal{K}(X) : t< \dim_\text{L} F \} \ \bigcap \ \Bigg( \bigcup_{n \in \mathbb{N}} \{F \in \mathcal{K}(X) : u-1/n< \dim_\text{L} F \}^c \Bigg)
\end{eqnarray*}
and it follows from Lemma \ref{borelset2} that this set is $\mathcal{F}_{\sigma \delta \sigma}$.
\end{proof}

As mentioned in Section \ref{questions}, we are currently unaware if $\Delta_L$ is Baire 2.  One possibility would be to prove the lower semicontinuity of the function $\Phi^0_{r,R}:\mathcal{K}(X) \to \mathbb{R}$ be defined by
\[
\Phi^0_{r,R}(F) = \inf_{x \in F} M_r \Big( B^0\big( x, R \big) \cap F \Big),
\]
but it seems unlikely to us that this function is lower semicontinuous.

\section{Proofs concerning quasi-self-similar sets} \label{proofs2}

\subsection{Proof of Theorem \ref{quasi}}  \label{proofquasi}

In this section we will prove Theorem \ref{quasi}.  Let $(X,d)$ be a metric space and let $F$ be a compact subset of $(X,d)$.  It follows immediately from the definition of box dimension that for all $\varepsilon, \rho >0$ there exists a constant $C_{\varepsilon, \rho} \geq 1$ such that for all $r \in (0, \rho]$ we have
\begin{equation} \label{simplebox0}
\tfrac{1}{C_{\varepsilon, \rho}} \, r^{-\underline{\dim}_\text{B} F + \varepsilon} \ \leq \  N_r(F) \  \leq \  C_{\varepsilon, \rho} \, r^{-\overline{\dim}_\text{B} F - \varepsilon}.
\end{equation}
For a map $f:A \to B$, for metric spaces $(A,d_A)$, $(B,d_B)$ we will write
\[
\text{Lip}^-(f) = \inf_{x,y \in A} \frac{d_B\big(f(x),f(y)\big)}{d_A(x,y)}.
\]
\emph{Proof of (1).}   Suppose $F$ satisfies (1) from Definition \ref{quasidef} with given parameters $a, r_0$ and write $s= \dim_\H F = \overline{\dim}_\text{B} F$.  Let $0<r<R\leq r_0/2$ and $x \in F$.  By condition (1) in the definition of quasi-self-similar, there exists an injection $g_1: B(x,r) \cap F \to F$ with $\text{Lip}^-(g_1) \geq  a \, (2R)^{-1}$.  If $\{U_i\}$ is an $ar/2R$ cover of $g_1(B(x,r) \cap F)$, then $\{g_1^{-1}(U_i)\}$ is an $r$ cover of $B(x,r) \cap F$.  It follows from this and (\ref{simplebox0}) that
\[
N_r\big(B(x,r) \cap F \big) \ \leq \ N_{ar/2R}\big(g_1(B(x,r) \cap F) \big) \ \leq \  N_{ar/2R}(F) \ \leq \ C_{\varepsilon,a/2} (2/a)^{s+\varepsilon} \, \bigg( \frac{R}{r} \bigg)^{s+\varepsilon}
\]
which gives that $\dim_\text{A} F \leq s+\varepsilon$ and letting $\varepsilon \to 0$ completes the proof.  \hfill \qed
\\ \\
\emph{Proof of (2).}
Suppose $F$ satisfies (2) from Definition \ref{quasidef} with given parameters $a, r_0$ and write $s= \dim_\H F = \underline{\dim}_\text{B} F$.  Let $0<r<R\leq r_0/2$ and $x \in F$.  By condition (2) in the definition of quasi-self-similar, there exists an injection $g_2: F \to B(x,r) \cap F $ with $\text{Lip}^-(g_2) \geq  a R$.  If $\{U_i\}$ is an $r$ cover of $g_2(F)$, then $\{g_2^{-1}(U_i)\}$ is an $r/aR$ cover of $F$.  It follows from this and (\ref{simplebox0}) that
\[
N_r\big(B(x,r) \cap F \big) \ \geq \ N_{r} \big(g_2(F) \big) \ \geq \  N_{r/aR}( F) \ \geq \ \tfrac{1}{C_{\varepsilon,1/a}} \, a^{s-\varepsilon} \, \bigg( \frac{R}{r} \bigg)^{s-\varepsilon}
\]
which gives that $\dim_\text{L} F \geq s-\varepsilon$ and letting $\varepsilon \to 0$ completes the proof.    \hfill \qed
\\ \\
\emph{Proof of (3).}  Let $F \subseteq (X,d)$ be a quasi-self-similar set with given parameters $a, r_0$ from Definition \ref{quasidef} and write $s= \dim_\H F$.  Note that it follows from (1)-(2) above that $\dim_\text{L} F = \dim_\text{A} F$, however it does not follow immediately that $F$ is Ahlfors regular, so we will prove that now.  It follows from the results in \cite{implicit} that
\begin{equation} \label{measestim}
a^s \ \leq \  \mathcal{H}^s(F) \  \leq \  4^s \, a^{-s}.
\end{equation}
Let $r \in (0, \,  r_0/2)$ and $x \in F$ and consider the set $B(x,r) \cap F := B(x,r) \cap F$.  By condition (1) in the definition of quasi-self-similar, there exists a map $g_1: B(x,r) \cap F \to F$ with $\text{Lip}^-(g_1) \geq  a \, (2r)^{-1}$.  It follows from this, (\ref{measestim}) and the scaling property for Hausdorff measure, that
\begin{equation} \label{upperahlfors}
\mathcal{H}^s(B(x,r) \cap F) \  \leq \ \text{Lip}^-(g_1)^{-s} \, \mathcal{H}^s\big(g_1(B(x,r) \cap F)\big)  \  \leq \ a^{-s} \, (2r)^{s} \, \mathcal{H}^s(F) \   \leq \ 8^s \, a^{-2s} \, r^s.
\end{equation}
Furthermore, by condition (2)  in the definition of quasi-self-similar, there exists a map $g_2: F \to B(x,r) \cap F$ with $\text{Lip}^-(g_2) \geq  a \, r$.  It follows from this, (\ref{measestim}) and the scaling property for Hausdorff measure, that
\begin{equation} \label{lowerahlfors}
\mathcal{H}^s(B(x,r) \cap F) \  \geq \ \mathcal{H}^s\big(g_2(F)\big)  \  \geq \ \text{Lip}^-(g_2)^{s} \, \mathcal{H}^s(F)  \  \geq \ a^{s} \, r^{s} \, \mathcal{H}^s(F) \   \geq \  a^{2s} \, r^s.
\end{equation}
It follows from (\ref{upperahlfors}) and (\ref{lowerahlfors}) that $F$ is locally Ahlfors regular setting $\lambda = 8^s \, a^{-2s}$ and since $F$ is compact we have that it is, in fact, Ahlfors regular. \hfill \qed

\section{Proofs concerning self-affine sets} \label{proofs3}

\subsection{Preliminary results and approximate squares } \label{prelim}

In this section we will introduce some notation and give some basic technical lemmas.  Let $F$ be a self-affine carpet, which is the attractor of an IFS $\{S_i\}_{i \in \mathcal{I}}$.  Write $\mathcal{I}^* = \bigcup_{k\geq1} \mathcal{I}^k$ to denote the set of all finite sequences with entries in $\mathcal{I}$ and for
\[
\textbf{\emph{i}}= \big(i_1, i_2, \dots, i_k \big) \in \mathcal{I}^*
\]
write
\[
S_{\textbf{\emph{i}}} = S_{i_1} \circ S_{i_2} \circ \dots \circ S_{i_k}
\]
and $\alpha_1 (\textbf{\emph{i}}) \geq \alpha_2 (\textbf{\emph{i}})$ for the singular values of the linear part of the map $S_{\textbf{\emph{i}}}$.  Note that, for all $\textbf{\emph{i}} \in \mathcal{I}^*$, the singular values, $\alpha_1 (\textbf{\emph{i}})$ and $\alpha_2 (\textbf{\emph{i}})$, are just the lengths of the sides of the rectangle $S_{\textbf{\emph{i}}}\big([0,1]^2\big)$.  Also, let
\[
\alpha_{\min} = \min \{\alpha_2(i) : i \in \mathcal{I} \}
\]
and
\[
\alpha_{\max} = \max \{\alpha_1(i) : i \in \mathcal{I} \}.
\]
Write $\mathcal{I}^\mathbb{N}$ to denote the set of all infinite $\mathcal{I}$-valued strings and for $\textbf{\emph{i}} \in \mathcal{I}^\mathbb{N}$ write $\textbf{\emph{i}}\rvert_k \in \mathcal{I}^k$ to denote the restriction of $\textbf{\emph{i}}$ to its first $k$ entries. Let $\Pi:\mathcal{I}^\mathbb{N} \to F$ be the natural surjection from the `symbolic' space to the `geometric' space defined by
\[
\Pi(\textbf{\emph{i}}) = \bigcap_{k \in \mathbb{N}} S_{\textbf{\emph{i}}\rvert_k}\big([0,1]^2\big).
\]
For $\textbf{\emph{i}}, \textbf{\emph{j}} \in \mathcal{I}^*$, we will write $\textbf{\emph{i}} \prec \textbf{\emph{j}}$ if $\textbf{\emph{j}} \vert_k = \textbf{\emph{i}}$ for some $k \leq \lvert \textbf{\emph{j}} \rvert$, where $\lvert \textbf{\emph{j}} \rvert$ is the length of the sequence $\textbf{\emph{j}}$.  For
\[
\textbf{\emph{i}} = (i_1,i_2, \dots, i_{k-1}, i_k)  \in \mathcal{I}^*
\]
let
\[
 \overline{\textbf{\emph{i}}} = (i_1,i_2, \dots, i_{k-1}) \in \mathcal{I}^* \cup \{ \omega\},
\]
where $\omega$ is the empty word.  Note that the map $S_\omega$ is taken to be the identity map, which has singular values both equal to 1.
\\ \\
A subset $\mathcal{I}_0 \subset \mathcal{I}^*$ is called a \emph{stopping} if for every $\textbf{\emph{i}} \in \mathcal{I}^*$ either there exists $\textbf{\emph{j}} \in \mathcal{I}_0$ such that $\textbf{\emph{i}} \prec \textbf{\emph{j}}$ or there exists a unique  $\textbf{\emph{j}} \in \mathcal{I}_0$ such that $\textbf{\emph{j}} \prec \textbf{\emph{i}}$.  An important class of stoppings will be ones where the members are chosen to have some sort of approximate property in common.  In particular, $r$-\emph{stoppings} are stoppings where the smallest sides of the corresponding rectangles are approximately equal to $r$.    For $r \in (0,1]$ we define the $r$-\emph{stopping}, $\mathcal{I}_r$, by
\[
\mathcal{I}_r = \big\{\textbf{\emph{i}} \in \mathcal{I}^* : \alpha_2(\textbf{\emph{i}}) < r \leq \alpha_2( \overline{\textbf{\emph{i}}}) \big\}.
\]
Note that for $\textbf{\emph{i}} \in \mathcal{I}_r$ we have
\begin{equation} \label{stop1}
\alpha_{\min} \, r \leq \alpha_2(\textbf{\emph{i}}) < r.
\end{equation}


We will now fix some notation for the dimensions of the various projections and slices we will be interested in.  Let
\[
s_1 = \dim_\text{B} \pi_1(F),
\]
\[
s_2 = \dim_\text{B} \pi_2(F),
\]
\[
t_1 = \max_{i \in \mathcal{I}} \ \dim_\text{B} \text{Slice}_{1,i}(F),
\]
\[
t_2 = \max_{i \in \mathcal{I}} \ \dim_\text{B} \text{Slice}_{2,i}(F),
\]
\[
u_1 = \min_{i \in \mathcal{I}} \ \dim_\text{B} \text{Slice}_{1,i}(F),
\]
and
\[
u_2 = \min_{i \in \mathcal{I}} \ \dim_\text{B} \text{Slice}_{2,i}(F).
\]
Note that all of these values can be easily computed as they are the dimensions of self-similar sets satisfying the open set condition.  We will be particularly interested in estimating the precise value of the covering function $N_r$ applied to the projections.  It follows immediately from the definition of box dimension that for all $\varepsilon>0$ there exists a constant $C_\varepsilon \geq 1$ such that for all $r \in (0, 1]$ we have
\begin{equation} \label{simplebox1}
\tfrac{1}{C_\varepsilon} \, r^{-s_1 + \varepsilon} \leq N_r(\pi_1 F) \leq C_\varepsilon \, r^{-s_1 - \varepsilon}
\end{equation}
and
\begin{equation} \label{simplebox2}
\tfrac{1}{C_\varepsilon} \, r^{-s_2 + \varepsilon} \leq N_r(\pi_2 F) \leq C_\varepsilon \, r^{-s_2 - \varepsilon}.
\end{equation}


Since the basic rectangles in the construction of $F$ often become very long and thin, they do not provide `natural' covers for $F$ unlike in the self-similar setting.  For this reason, we need to introduce \emph{approximate squares}, which are now a standard concept in the study of self-affine carpets.  The basic idea is to group together the construction rectangles into collections which look roughly like a square.  Let $\textbf{\emph{i}} \in \mathcal{I}^{\mathbb{N}}$ and $r>0$.  Let $k_1(\textbf{\emph{i}},r)$ equal the unique number $k \in \mathbb{N}$ such that
\[
c_{\textbf{\emph{i}}\rvert_{k+1}}< r \leq c_{\textbf{\emph{i}}\rvert_{k}}
\]
and 
$k_2(\textbf{\emph{i}},r)$ equal the unique number $k \in \mathbb{N}$ such that
\[
d_{\textbf{\emph{i}}\rvert_{k+1}}< r \leq d_{\textbf{\emph{i}}\rvert_{k}}.
\]
Finally, we define the approximate square $Q(\textbf{\emph{i}},r)$ `centered' at $\Pi(\textbf{\emph{i}})$, with `radius' $r$ in the following way.  If $k_1(\textbf{\emph{i}},r) < k_2(\textbf{\emph{i}},r)$, then
\[
Q(\textbf{\emph{i}},r) = S_{\textbf{\emph{i}}\rvert_{k_1(\textbf{\emph{i}},r)}}\big([0,1]^2\big) \ \cap \ \Big\{x \in [0,1]^{2} : \pi_1(x)  \in \pi_1\big( S_{\textbf{\emph{i}}\rvert_{k_2(\textbf{\emph{i}},r)}}\big([0,1]^2\big)  \big) \Big\},
\]
if $k_1(\textbf{\emph{i}},r) > k_2(\textbf{\emph{i}},r)$, then
\[
Q(\textbf{\emph{i}},r) = S_{\textbf{\emph{i}}\rvert_{k_2(\textbf{\emph{i}},r)}}\big([0,1]^2\big) \ \cap \ \Big\{x \in [0,1]^{2} : \pi_2(x)  \in \pi_2\big( S_{\textbf{\emph{i}}\rvert_{k_1(\textbf{\emph{i}},r)}}\big([0,1]^2\big)  \big) \Big\},
\]
and if $k_1(\textbf{\emph{i}},r) = k_2(\textbf{\emph{i}},r) = k$, then
\[
Q(\textbf{\emph{i}},r) = S_{\textbf{\emph{i}}\rvert_k}\big([0,1]^2\big) .
\]
We will write
\[
\mathcal{I}_{Q(\textbf{\emph{i}},r)} = \big\{ \textbf{\emph{j}} \in \mathcal{I}^{\max\{k_1(\textbf{\emph{i}},r), \,  k_2(\textbf{\emph{i}},r) \}}: S_\textbf{\emph{j}}(F) \subseteq Q(\textbf{\emph{i}},r) \big\}.
\]
The following lemma gives some of the basic properties of approximate squares.
\begin{lma} \label{cubes} Let $\textbf{i} \in \mathcal{I}^{\mathbb{N}}$ and $r>0$.  
\begin{itemize}
\item[(1)] If $k_1(\textbf{i},r) \geq k_2(\textbf{i},r)$, then for $\textbf{j} \in \mathcal{I}_{Q(\textbf{i},r)}$ we have
\[
r \, \leq \, c_\textbf{j}\,  \leq \,  c_{\max}^{-1} \, r.
\]
\item[(2)] If $k_1(\textbf{i},r) \leq k_2(\textbf{i},r)$, then for $\textbf{j} \in \mathcal{I}_{Q(\textbf{i},r)}$ we have
\[
r \, \leq \, d_\textbf{j}\,  \leq \,  d_{\max}^{-1} \, r.
\]
\item[(3)]  The approximate square $Q(\textbf{i},r)$ is a rectangle with sides parallel to the coordinate axes and with base length in the interval $[r,c_{\max}^{-1} \, r]$ and height in the interval $[r,d_{\max}^{-1} \, r]$, so is indeed \emph{approximately} a square.
\item[(4)] We have
\[
Q(\textbf{i},r)  \subset B\big(\Pi(\textbf{i}), \sqrt{2} \, \alpha_{\min}^{-1} \, r\big).
\]
\item[(5)] For any $x \in F$, the ball $B(x,r)$ can be covered by at most 9 approximate squares of radius $r$ and the constant 9 is sharp.
\end{itemize}
\end{lma}

\begin{proof}
These facts follow immediately from the definition of approximate squares and are omitted.
\end{proof}

Note that (4) and (5) together imply that we may replace $N_r(B(x,R))$ with $N_r(Q(\textbf{\emph{i}},R))$ in the definitions of Assouad and lower dimension.  Bara\'nski \cite{baranski} defined the numbers $D_A, D_B$ to be the unique real numbers satisfying
\[
\sum_{i \in \mathcal{I}} c_i^{s_1} d_i^{D_A-s_1} = 1 \qquad \text{and} \qquad \sum_{i \in \mathcal{I}} d_i^{s_2} c_i^{D_B-s_2} = 1
\]
respectively.  He then proved that $\dim_\text{B} F = \max \{D_A, \, D_B\}$.  The following lemma relates the numbers $D_A$ and $D_B$ to the numbers $s_1$, $s_2$, $u_1$, $u_2$, $t_1$ and $t_2$.

\begin{lma} \label{DADB}
We have
\[
s_1+u_1 \ \leq \ D_A  \ \leq \ s_1+t_1
\]
and
\[
s_2+u_2 \ \leq \ D_B  \ \leq \ s_2+t_2.
\]
\end{lma}

\begin{proof}
We will prove that $s_1+u_1 \leq  D_A$.  The other inequalities are proved similarly.  Suppose that $D_A < s_1+u_1$.  Write $m$ and $n$ for the number of non-empty columns and rows respectively, counting columns from the left and rows from the bottom, and for $i \in \{1, \dots, m\}$ write
\[
\mathcal{C}_i = \{j \in \mathcal{I} : S_j([0,1]^2) \text{ is found in the $i$th non-empty column of the defining pattern} \}
\]
and for $i \in \{1, \dots, n\}$ write
\[
\mathcal{R}_i = \{j \in \mathcal{I} : S_j([0,1]^2) \text{ is found in the $i$th non-empty row of the defining pattern} \}.
\]
A useful consequence of splitting $\mathcal{I}$ up into columns and rows is that if $i,j$ are in the same column, then $c_i = c_j$ and if $i,j$ are in the same row, then $d_i = d_j$.  As such, for $i \in \{1, \dots, m\}$ we will write $\hat{c}_i$ for the common base length in the $i$th column and for $i \in \{1, \dots, n\}$ we will write $\hat{d}_i$ for the common height in the $i$th row.  We have
\[
1 \ = \  \sum_{i \in \mathcal{I}} c_i^{s_1} d_i^{D_A-s_1}  \ > \    \sum_{i \in \mathcal{I}} c_i^{s_1} d_i^{s_1+u_1-s_1} \  = \  \sum_{i=1}^{m} \hat{c}_i^{s_1} \sum_{j \in \mathcal{C}_i} d_j^{u_1} \  \geq \   \sum_{i=1}^{m} \hat{c}_i^{s_1} \ = \ 1
\]
which is a contradiction.
\end{proof}

\begin{lma} \label{DADBadditive}
Let $\mathcal{I}_0$ be a stopping.  Then
\[
\sum_{\textbf{i} \in \mathcal{I}_0} c_\textbf{i}^{s_1} d_\textbf{i}^{D_A-s_1} =  \sum_{\textbf{i} \in \mathcal{I}_0} d_\textbf{i}^{s_2} c_\textbf{i}^{D_B-s_2} = 1
\]
\end{lma}

\begin{proof}
This follows immediately from the definitions of $D_A$ and $D_B$.
\end{proof}

Let $r>0$ and $\textbf{\emph{i}} =(i_1, i_2, \dots)  \in \mathcal{I}^\mathbb{N}$. We call $\mathcal{I}_0 \subset \mathcal{I}^*$ a $Q(\textbf{\emph{i}},r)$-\emph{pseudo stopping} if the following conditions are satisfied.
\begin{itemize}
\item[(1)] For each $\textbf{\emph{j}} \in \mathcal{I}_0$, we have $\textbf{\emph{i}}_{\min\{k_1(\textbf{\emph{i}},r), k_2(\textbf{\emph{i}},r)\}} \prec \textbf{\emph{j}}$,
\item[(2)] For each $\textbf{\emph{j}} \in \mathcal{I}_0$, we have $\lvert \textbf{\emph{j}} \rvert \leq \max\{k_1(\textbf{\emph{i}},r), k_2(\textbf{\emph{i}},r)\}$,
\item[(3)] For every $\textbf{\emph{i}}' \in \mathcal{I}^{\max\{k_1(\textbf{\emph{i}},r), k_1(\textbf{\emph{i}},r)\}}$ there exists a unique $\textbf{\emph{j}} \in \mathcal{I}_0$ such that $\textbf{\emph{j}} \prec \textbf{\emph{i}}'$.
\end{itemize}
The important feature of a $Q(\textbf{\emph{i}},r)$-pseudo stopping, $\mathcal{I}_0$, is that the sets $\{S_\textbf{\emph{j}}([0,1]^2) \}_{\textbf{\emph{j}} \in \mathcal{I}_0}$ intersect the approximate square $Q(\textbf{\emph{i}},r)$ in such a way as to induce natural IFSs of similarities on $[0,1]$.  For instance, if $\max\{k_1(\textbf{\emph{i}},r), k_2(\textbf{\emph{i}},r)\} = k_1(\textbf{\emph{i}},r)$, then each of the base lengths of the sets $\{S_\textbf{\emph{j}}([0,1]^2) \}_{\textbf{\emph{j}} \in \mathcal{I}_0}$ are greater than or equal to the base length of the approximate square.  We then focus on the vertical lengths and, after scaling these up by the height of $Q(\textbf{\emph{i}},r)$, use these as similarity ratios for a set of 1-dimensional contractions on $[0,1]$.  It is easy to see that in this `vertical case', the similarity dimension of the induced IFS lies in the interval $[u_1, t_1]$.  This trick is illustrated in the following lemma and will be used frequently in the subsequent proofs.

\begin{lma} \label{selfsim}
Let $r>0$, $\textbf{i} \in \mathcal{I}^\mathbb{N}$ and let $\mathcal{I}_0$ be a $Q(\textbf{i},r)$-pseudo stopping and assume that $k_1(\textbf{i},r) \geq k_2(\textbf{i},r)$.  Then, for any $t \geq t_1$, we have
\[
\sum_{\textbf{j} \in \mathcal{I}_0 } (d_\textbf{j}/r)^{t} \ \leq \  d_{\min}^{-t}
\]
and for any $u \leq u_1$, we have
\[
\sum_{\textbf{j} \in \mathcal{I}_0 } (d_\textbf{j}/r)^{u} \ \geq \  1.
\]
\end{lma}

\begin{proof}
This proof is straightforward and we will only sketch it.  Let $t \geq t_1$ and let $k_1 = k_1(\textbf{\emph{i}},r)  \geq k_2(\textbf{\emph{i}},r) = k_2$.  We have
\begin{eqnarray*}
\sum_{\textbf{\emph{j}} \in \mathcal{I}_0} (d_\textbf{\emph{j}}/r)^{t} &=& \sum_{\textbf{\emph{j}} \in \mathcal{I}_0} (d_{i_1} \dots d_{j_{k_2}}/r)^{t} \, (d_{j_{k_2+1}} \dots d_{j_{k_1}})^{t} \\ \\
&\leq& d_{\min}^{-t} \, \sum_{\textbf{\emph{j}} \in \mathcal{I}_0}  (d_{j_{k_2+1}} \dots d_{j_{k_1}})^{t}  \qquad \text{by Lemma \ref{cubes} (2)}\\ \\
&\leq& d_{\min}^{-t}
\end{eqnarray*}
since viewing the $(d_{j_{k_2+1}} \dots d_{j_{k_1}})$ as contraction ratios of a 1-dimensional IFS of similarities and noting that this IFS has similarity dimension less than or equal to $t_1$, yields
\[
\sum_{\textbf{\emph{i}} \in \mathcal{I}_0}  (d_{j_{k_2+1}} \dots d_{j_{k_1}})^{t} \ \leq  \ 1.
\]
The second estimate is similar.  For $u \leq u_1$, we have
\begin{eqnarray*}
\sum_{\textbf{\emph{j}} \in \mathcal{I}_0} (d_\textbf{\emph{j}}/r)^{u} &=& \sum_{\textbf{\emph{j}} \in \mathcal{I}_0} (d_{j_1} \dots d_{j_{k_2}}/r)^{u} \, (d_{j_{k_2+1}} \dots d_{j_{k_1}})^{u} \\ \\
&\geq&  \sum_{\textbf{\emph{i}} \in \mathcal{I}_0}  (d_{j_{k_2+1}} \dots d_{j_{k_1}})^{u}  \qquad \text{by Lemma \ref{cubes} (2)} \\ \\
&\geq& 1
\end{eqnarray*}
since viewing the $(d_{j_{k_2+1}} \dots d_{j_{k_1}})$ as contraction ratios of a 1-dimensional IFS of similarities and noting that this IFS has similarity dimension greater than or equal to $u_1$, yields
\[
 \sum_{\textbf{\emph{j}} \in \mathcal{I}_0}  (d_{j_{k_2+1}} \dots d_{j_{k_1}})^{u} \\ \\
 \ \geq \  1.
\]
This completes the proof.
\end{proof}

Note that there are obvious analogues of the above Lemma in the case $k_1(\textbf{\emph{i}},r)  < k_2(\textbf{\emph{i}},r) $, but we omit them here.  Two natural examples of $Q(\textbf{\emph{i}},r)$-pseudo stoppings are $\{\textbf{\emph{i}} \vert_{\min\{k_1(\textbf{\emph{i}},r), \, k_2(\textbf{\emph{i}},r)\}}\}$ and $\mathcal{I}_{Q(\textbf{\emph{i}},r)}$.  We give an example of an intermediate $Q(\textbf{\emph{i}},r)$-pseudo stopping in the following figure.

\begin{figure}[H]
	\centering
	\includegraphics[width=155mm]{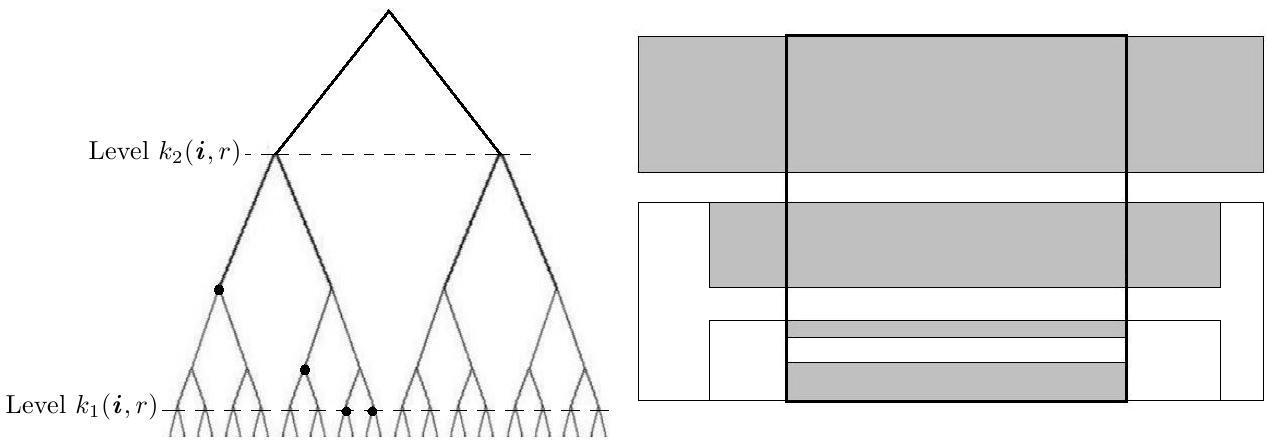}
\caption{A binary tree giving a graphical representation of a pseudo stopping, $\mathcal{I}_0$, with black dots representing the elements of the pseudo stopping (left) and an indication of how the grey rectangles $\{S_\textbf{\emph{j}}([0,1]^2) \}_{\textbf{\emph{j}} \in \mathcal{I}_0}$ intersect the approximate square $Q(\textbf{\emph{i}},r)$ (right).}
\end{figure}

\subsection{Proof of Theorem \ref{uppermain} for the mixed class} \label{upperbaranski}

\textbf{Upper bound.}  The key to proving the upper bound for $\dim_\text{A}F$ is to find the appropriate way to cover approximate squares.  Fix $\textbf{\emph{i}}' \in \mathcal{I}^\mathbb{N}$, $R>0$ and $r \in (0, R)$ and consider the approximate square $Q(\textbf{\emph{i}}',R)$.  Without loss of generality assume that $k_1(\textbf{\emph{i}}',R) \geq k_2(\textbf{\emph{i}}',R)$ and to simplify notation let $k=k_1(\textbf{\emph{i}}',R)$.   Furthermore we may assume that there exists $j_1, j_2 \in \mathcal{I}$ such that $c_{j_1} > d_{j_1}$ and $c_{j_2} < d_{j_2}$ as otherwise we are in the horizontal or vertical class, which will be dealt with in Section \ref{upperlalley}. Let $s =  \max_{i \in \mathcal{I}} \, \max_{j=1,2} \, \big(\dim_\text{B} \pi_j(F) \, + \,  \dim_\text{B}  \text{Slice}_{j,i}(F) \big)$.  It suffices to prove that for all $\varepsilon \in (0,1)$, there exists a constant $C(\varepsilon)$ such that
\[
N_r\Big( Q(\textbf{\emph{i}}',R) \cap F \Big) \leq C(\varepsilon) \, \bigg(\frac{R}{r} \bigg)^{ s +\varepsilon}.
\]
Let $\varepsilon \in (0,1)$. Writing
\[
\mathcal{I}_Q = \mathcal{I}_{Q(\textbf{\emph{i}}',R)} = \big\{ \textbf{\emph{j}} \in \mathcal{I}^k: S_\textbf{\emph{j}}(F) \subseteq Q(\textbf{\emph{i}}',R) \big\},
\]
we first split the approximate square $Q(\textbf{\emph{i}}',R)$ up as
\[
Q(\textbf{\emph{i}}',R) \cap F = \bigcup_{ \textbf{\emph{i}} \in \mathcal{I}_Q}  S_\textbf{\emph{i}}(F).
\]
Secondly, we group together the sets $S_\textbf{\emph{i}}(F)$ for which $d_\textbf{\emph{i}}<r$ and cover their union separately.  Within the other sets, $S_\textbf{\emph{i}}(F)$ we iterate the IFS until one side of the rectangle $S_\textbf{\emph{ij}}\big( [0,1]^2 \big) \supseteq S_\textbf{\emph{ij}}(F)$ is smaller than $r$.  This is reminiscent of the techniques used by the author in \cite{me_box}.  We then cover each of the resulting copies of $F$ individually.  This is especially convenient because covering the part of $F$ which lies in such a rectangle by sets of radius $r$ is equivalent to covering a scaled down copy of the projection of $F$ onto either the horizontal or vertical axis.  Finally, we split the sets $S_\textbf{\emph{ij}}(F)$ which we are covering individually into two groups according to whether the short side of $S_\textbf{\emph{ij}}\big( [0,1]^2 \big)$ is vertical or horizontal.  We have
\begin{eqnarray}
N_r\Big( Q(\textbf{\emph{i}}',R) \cap F \Big) &\leq& N_r\Bigg( \ \ \bigcup_{\substack{\textbf{\emph{i}} \in \mathcal{I}_Q : \\ \\
d_\textbf{\emph{i}} < r}}  S_{\textbf{\emph{i}}}(F) \  \ \Bigg) \ + \ \sum_{\substack{\textbf{\emph{i}} \in \mathcal{I}_Q : \\ \\
d_\textbf{\emph{i}} \geq r}}  \ \sum_{\substack{\textbf{\emph{j}} \in \mathcal{I}^*:\\ \\
\textbf{\emph{ij}} \in \mathcal{I}_r} }N_r \Big( S_{\textbf{\emph{ij}}}(F)  \Big) \nonumber \\ \nonumber \\
&\leq& N_r\Bigg( \ \ \bigcup_{\substack{\textbf{\emph{i}} \in \mathcal{I}_Q :  \\ \\
d_\textbf{\emph{i}} < r}}  S_{\textbf{\emph{i}}}(F) \  \ \Bigg) \nonumber \\ \nonumber \\
&\,&  + \ \sum_{\substack{\textbf{\emph{i}} \in \mathcal{I}_Q : \\ \\
d_\textbf{\emph{i}} \geq r}}   \ \sum_{\substack{\textbf{\emph{j}} \in \mathcal{I}^*:\\ \\
\textbf{\emph{ij}} \in \mathcal{I}_r, \\ \\
\alpha_2(\textbf{\emph{ij}}) = d_\textbf{\emph{ij}}}} N_r \Big( S_{\textbf{\emph{ij}}}(F)  \Big) \ + \  \sum_{\substack{\textbf{\emph{i}} \in \mathcal{I}_Q \\ \\
d_\textbf{\emph{i}} \geq r}}  \ \sum_{\substack{\textbf{\emph{j}} \in \mathcal{I}^*:\\ \\
\textbf{\emph{ij}} \in \mathcal{I}_r, \\ \\
\alpha_2(\textbf{\emph{ij}}) = c_\textbf{\emph{ij}}}} N_r \Big( S_{\textbf{\emph{ij}}}(F)  \Big) \nonumber \\ \nonumber \\
&=& N_r\Bigg( \ \ \bigcup_{\substack{\textbf{\emph{i}} \in \mathcal{I}_Q :\\ \\
d_\textbf{\emph{i}} < r}}  S_{\textbf{\emph{i}}}(F) \  \ \Bigg) \nonumber \\ \nonumber \\
&\,&  + \  \sum_{\substack{\textbf{\emph{i}} \in \mathcal{I}_Q :\\ \\
d_\textbf{\emph{i}} \geq r}}  \ \sum_{\substack{\textbf{\emph{j}} \in \mathcal{I}^*:\\ \\
\textbf{\emph{ij}} \in \mathcal{I}_r, \\ \\
\alpha_2(\textbf{\emph{ij}}) = d_\textbf{\emph{ij}}}} N_{r/c_\textbf{\emph{ij}}} \big( \pi_1(F)  \big) \ + \  \sum_{\substack{\textbf{\emph{i}} \in \mathcal{I}_Q :\\ \\
d_\textbf{\emph{i}} \geq r}}  \ \sum_{\substack{\textbf{\emph{j}} \in \mathcal{I}^*:\\ \\
\textbf{\emph{ij}} \in \mathcal{I}_r, \\ \\
\alpha_2(\textbf{\emph{ij}}) = c_\textbf{\emph{ij}}}} N_{r/d_\textbf{\emph{ij}}} \big( \pi_2(F)  \big) \nonumber
\end{eqnarray}
Now that we have established a natural way to cover $Q(\textbf{\emph{i}}',R)$, we need to show that this yields the correct estimates. We will analyse each of the three above terms separately.  Write
\[
\mathcal{I}_Q^{<r} = \big\{\textbf{\emph{i}} \in \mathcal{I}_Q : d_\textbf{\emph{i}} < r \big\}.
\]
For the first term, we have

\begin{eqnarray*}
N_r\Bigg( \ \ \bigcup_{\substack{\textbf{\emph{i}} \in \mathcal{I}_Q :\\ \\
d_\textbf{\emph{i}} < r}}  S_{\textbf{\emph{i}}}(F) \  \ \Bigg) & = & N_r\Bigg( \ \ \bigcup_{\textbf{\emph{i}} \in \mathcal{I}_Q^{<r}}  S_{\textbf{\emph{i}}}(F) \  \ \Bigg) \\ \\
& = & N_r\Bigg( \ \ \bigcup_{\substack{\textbf{\emph{j}} \in \mathcal{I}_r :\\ \\
\exists \textbf{\emph{i}}  \in \mathcal{I}_Q^{<r} , \  \textbf{\emph{j}}  \prec \textbf{\emph{i}} }}  S_{\textbf{\emph{i}}}(F) \cap  Q(\textbf{\emph{i}}',R) \  \ \Bigg) \\ \\
& \leq &  \sum_{\substack{\textbf{\emph{j}} \in \mathcal{I}_r :\\ \\
\exists \textbf{\emph{i}}  \in \mathcal{I}_Q^{<r} , \  \textbf{\emph{j}}  \prec \textbf{\emph{i}} }}  N_r \Big( S_{\textbf{\emph{i}}}(F) \cap  Q(\textbf{\emph{i}}',R) \Big) \\ \\
& = &  \sum_{\substack{\textbf{\emph{j}} \in \mathcal{I}_r :\\ \\
\exists \textbf{\emph{i}}  \in \mathcal{I}_Q^{<r} , \  \textbf{\emph{j}}  \prec \textbf{\emph{i}} }}  N_{r/c_{\textbf{\emph{i}}}} \big( \pi_1(F) \big) \\ \\
& \leq &  \sum_{\substack{\textbf{\emph{j}} \in \mathcal{I}_r :\\ \\
\exists \textbf{\emph{i}}  \in \mathcal{I}_Q^{<r} , \  \textbf{\emph{j}}  \prec \textbf{\emph{i}} }}  C_\varepsilon \bigg( \frac{c_\textbf{\emph{i}}}{r}\bigg)^{s_1+\varepsilon} \qquad \qquad \text{by (\ref{simplebox1})} \\ \\
& \leq & C_\varepsilon \, c_{\max}^{-2} \,  \bigg( \frac{R}{r}\bigg)^{s_1+t_1+\varepsilon} \sum_{\substack{\textbf{\emph{j}} \in \mathcal{I}_r :\\ \\
\exists \textbf{\emph{i}}  \in \mathcal{I}_Q^{<r} , \  \textbf{\emph{j}}  \prec \textbf{\emph{i}} }}   (r/R)^{t_1} \qquad \qquad  \text{by Lemma \ref{cubes}(1)}  \\ \\
& \leq & C_\varepsilon \, c_{\max}^{-2} \, \alpha_{\min}^{-1}\,  \bigg( \frac{R}{r}\bigg)^{s_1+t_1+\varepsilon} \sum_{\substack{\textbf{\emph{j}} \in \mathcal{I}_r :\\ \\
\exists \textbf{\emph{i}}  \in \mathcal{I}_Q^{<r} , \  \textbf{\emph{j}}  \prec \textbf{\emph{i}} }}   (d_\textbf{\emph{j}}/R)^{t_1} \qquad \qquad  \text{by (\ref{stop1})} \\ \\
& \leq & C_\varepsilon \, c_{\max}^{-2} \, \alpha_{\min}^{-1} \, d_{\min}^{-t_1} \, \bigg( \frac{R}{r}\bigg)^{s_1+t_1+\varepsilon}
\end{eqnarray*}
by Lemma \ref{selfsim} since $\{\textbf{\emph{j}} \in \mathcal{I}_r : \text{ there exists } \textbf{\emph{i}}  \in \mathcal{I}_Q^{<r} \text{ such that } \  \textbf{\emph{j}}  \prec \textbf{\emph{i}}\}$ is clearly contained in some $Q(\textbf{\emph{i}}',R)$-pseudo stopping.
For the second term, by (\ref{simplebox1}), we have
\begin{eqnarray*}
\sum_{\substack{\textbf{\emph{i}} \in \mathcal{I}_Q: \\ \\
d_\textbf{\emph{i}} \geq r}}  \ \sum_{\substack{\textbf{\emph{j}} \in \mathcal{I}^*:\\ \\
\textbf{\emph{ij}} \in \mathcal{I}_r, \\ \\
\alpha_2(\textbf{\emph{ij}}) = d_\textbf{\emph{ij}}}} N_{r/c_\textbf{\emph{ij}}} \big( \pi_1(F)  \big)  &\leq& \sum_{\substack{\textbf{\emph{i}} \in \mathcal{I}_Q:\\ \\
d_\textbf{\emph{i}} \geq r}}  \ \sum_{\substack{\textbf{\emph{j}} \in \mathcal{I}^*:\\ \\
\textbf{\emph{ij}} \in \mathcal{I}_r, \\ \\
\alpha_2(\textbf{\emph{ij}}) = d_\textbf{\emph{ij}}}} C_\varepsilon \, \Big(\frac{c_\textbf{\emph{i}}c_\textbf{\emph{j}}}{r} \Big)^{s_1+\varepsilon} \\ \\
&\leq& C_\varepsilon \, \bigg(\frac{1}{r} \bigg)^{s_1+\varepsilon} \, \sum_{\substack{\textbf{\emph{i}} \in \mathcal{I}_Q: \\ \\
d_\textbf{\emph{i}} \geq r}}  \ c_\textbf{\emph{i}}^{s_1+\varepsilon} \sum_{\substack{\textbf{\emph{j}} \in \mathcal{I}^*:\\ \\
\textbf{\emph{ij}} \in \mathcal{I}_r, \\ \\
\alpha_2(\textbf{\emph{ij}}) = d_\textbf{\emph{ij}}}} c_\textbf{\emph{j}}^{s_1} \\ \\
&\leq& C_\varepsilon \,  \bigg(\frac{1}{r} \bigg)^{s_1+\varepsilon} \, \sum_{\substack{\textbf{\emph{i}} \in \mathcal{I}_Q:\\ \\
d_\textbf{\emph{i}} \geq r}} (c_{\max}^{-1}\,  R)^{s_1+\varepsilon} \  \sum_{\substack{\textbf{\emph{j}} \in \mathcal{I}^*:\\ \\
\textbf{\emph{ij}} \in \mathcal{I}_r, \\ \\
\alpha_2(\textbf{\emph{ij}}) = d_\textbf{\emph{ij}}}} c_\textbf{\emph{j}}^{s_1} \,  \Big(d_\textbf{\emph{i}} \, d_\textbf{\emph{j}}\, r^{-1} \, \alpha_{\min}^{-1}\Big)^{t_1} \\ \\
&\,&  \qquad \qquad \qquad \qquad \qquad \qquad   \qquad \qquad  \text{by Lemma \ref{cubes}(1)  and (\ref{stop1})}\\ \\
&\leq& C_\varepsilon \, c_{\max}^{-2} \, \alpha_{\min}^{-1} \,  \bigg(\frac{R}{r} \bigg)^{s_1+\varepsilon} \,  \bigg(\frac{1}{r} \bigg)^{t_1} \,  \sum_{\substack{\textbf{\emph{i}} \in \mathcal{I}_Q: \\ \\
d_\textbf{\emph{i}} \geq r}} d_\textbf{\emph{i}}^{t_1}  \  \sum_{\substack{\textbf{\emph{j}} \in \mathcal{I}^*:\\ \\
\textbf{\emph{ij}} \in \mathcal{I}_r, \\ \\
\alpha_2(\textbf{\emph{ij}}) = d_\textbf{\emph{ij}}}} c_\textbf{\emph{j}}^{s_1} \,  d_\textbf{\emph{j}}^{t_1}  \\ \\
&\leq& C_\varepsilon \, c_{\max}^{-2} \,\alpha_{\min}^{-1} \, \bigg(\frac{R}{r} \bigg)^{s_1+t_1+\varepsilon}  \,  \sum_{\substack{\textbf{\emph{i}} \in \mathcal{I}_Q}} (d_\textbf{\emph{i}}/R)^{t_1}  \ \sum_{\substack{\textbf{\emph{j}} \in \mathcal{I}^*:\\ \\
\textbf{\emph{ij}} \in \mathcal{I}_r
}} c_\textbf{\emph{j}}^{s_1} \,  d_\textbf{\emph{j}}^{D_A-s_1}  \\ \\
&\,&  \qquad \qquad \qquad \qquad \qquad \qquad   \qquad \qquad  \text{by Lemma \ref{DADB}}\\ \\
&\leq& C_\varepsilon \, c_{\max}^{-2} \,\alpha_{\min}^{-1} \, \bigg(\frac{R}{r} \bigg)^{s_1+t_1+\varepsilon}  \,  \sum_{\substack{\textbf{\emph{i}} \in \mathcal{I}_Q}} (d_\textbf{\emph{i}}/R)^{t_1} \qquad \qquad  \text{by Lemma \ref{DADBadditive}}\\ \\
&\leq& C_\varepsilon \, c_{\max}^{-2} \,\alpha_{\min}^{-1} \, d_{\min}^{-t_1} \, \bigg(\frac{R}{r} \bigg)^{s_1+t_1+\varepsilon} 
\end{eqnarray*}
by Lemma \ref{selfsim} since $\mathcal{I}_Q$ is a $Q(\textbf{\emph{i}}',R)$-pseudo stopping.  Finally, for the third term, by (\ref{simplebox2}), we have
\begin{eqnarray*}
\sum_{\substack{\textbf{\emph{i}} \in \mathcal{I}_Q: \\ \\
d_\textbf{\emph{i}} \geq r}}  \ \sum_{\substack{\textbf{\emph{j}} \in \mathcal{I}^*:\\ \\
\textbf{\emph{ij}} \in \mathcal{I}_r, \\ \\
\alpha_2(\textbf{\emph{ij}}) = c_\textbf{\emph{ij}}}} N_{r/d_\textbf{\emph{ij}}} \big( \pi_2(F)  \big) &\leq& \sum_{\substack{\textbf{\emph{i}} \in \mathcal{I}_Q }}   \ \sum_{\substack{\textbf{\emph{j}} \in \mathcal{I}^*:\\ \\
\textbf{\emph{ij}} \in \mathcal{I}_r, \\ \\
\alpha_2(\textbf{\emph{ij}}) = c_\textbf{\emph{ij}}}}  C_\varepsilon \, \Big(\frac{d_\textbf{\emph{i}}d_\textbf{\emph{j}}}{r} \Big)^{s_2+\varepsilon} \\ \\
&\leq& C_\varepsilon \, \bigg(\frac{1}{r} \bigg)^{s_2+\varepsilon} \, \sum_{\substack{\textbf{\emph{i}} \in \mathcal{I}_Q }}  \ d_\textbf{\emph{i}}^{s_2+\varepsilon}  \ \sum_{\substack{\textbf{\emph{j}} \in \mathcal{I}^*:\\ \\
\textbf{\emph{ij}} \in \mathcal{I}_r, \\ \\
\alpha_2(\textbf{\emph{ij}}) = c_\textbf{\emph{ij}}}}  d_\textbf{\emph{j}}^{s_2} \\ \\
&\leq& C_\varepsilon \, \bigg(\frac{1}{r} \bigg)^{s_2+\varepsilon} \, \sum_{\substack{\textbf{\emph{i}} \in \mathcal{I}_Q }}  \ d_\textbf{\emph{i}}^{s_2+\varepsilon}  \ \sum_{\substack{\textbf{\emph{j}} \in \mathcal{I}^*:\\ \\
\textbf{\emph{ij}} \in \mathcal{I}_r, \\ \\
\alpha_2(\textbf{\emph{ij}}) = c_\textbf{\emph{ij}}}} d_\textbf{\emph{j}}^{s_2} \,  \Big(c_\textbf{\emph{j}} \, c_\textbf{\emph{i}} \, r^{-1} \,  \alpha_{\min}^{-1}\Big)^{t_2}   \qquad \text{by (\ref{stop1})}\\ \\
&=& C_\varepsilon \,  \alpha_{\min}^{-1} \, \bigg(\frac{1}{r} \bigg)^{s_2+\varepsilon} \,  \bigg(\frac{1}{r} \bigg)^{t_2} \,  \sum_{\substack{\textbf{\emph{i}} \in \mathcal{I}_Q}}  d_\textbf{\emph{i}}^{s_2+\varepsilon}  \, c_\textbf{\emph{i}}^{t_2}  \ \sum_{\substack{\textbf{\emph{j}} \in \mathcal{I}^*:\\ \\
\textbf{\emph{ij}} \in \mathcal{I}_r, \\ \\
\alpha_2(\textbf{\emph{ij}}) = c_\textbf{\emph{ij}}}}  d_\textbf{\emph{j}}^{s_2} \,  c_\textbf{\emph{j}}^{t_2}  \\ \\
&\leq& C_\varepsilon \,  \alpha_{\min}^{-1} \, \bigg(\frac{1}{r} \bigg)^{s_2+t_2+\varepsilon}  \, R^{s_2+\varepsilon} \,  \sum_{\textbf{\emph{i}} \in \mathcal{I}_Q} (d_\textbf{\emph{i}}/R)^{s_2}\,  (c_{\max}^{-1}R)^{t_2}  \  \sum_{\substack{\textbf{\emph{j}} \in \mathcal{I}^*:\\ \\
\textbf{\emph{ij}} \in \mathcal{I}_r}}  d_\textbf{\emph{j}}^{s_2} \,  c_\textbf{\emph{j}}^{D_B-s_2}  \\ \\
&\,&  \qquad \qquad \qquad \qquad  \qquad \qquad \qquad \qquad     \text{by Lemma \ref{cubes}(1) and Lemma \ref{DADB}} \\ \\
&\leq& C_\varepsilon \, c_{\max}^{-2}\,  \alpha_{\min}^{-1}  \, \bigg(\frac{R}{r} \bigg)^{s_2+t_2+\varepsilon}  \,  \sum_{\substack{\textbf{\emph{i}} \in \mathcal{I}_Q}} (d_\textbf{\emph{i}}/R)^{s_2} \qquad \qquad \text{by Lemma \ref{DADBadditive}}\\ \\
&\leq& C_\varepsilon \, c_{\max}^{-2}\,  \alpha_{\min}^{-1}  \, d_{\min}^{-s_2} \, \bigg(\frac{R}{r} \bigg)^{s_2+t_2+\varepsilon} 
\end{eqnarray*}
by Lemma \ref{selfsim}  since $\mathcal{I}_Q$ is a $Q(\textbf{\emph{i}}',R)$-pseudo stopping and $s_2 \geq t_1$.  Since both $s_1+t_1$ and $s_2+t_2$ are less than or equal to $s$, combining the above estimates for the three terms appearing in the natural cover for $Q(\textbf{\emph{i}}',R)$ which were introduced at the beginning of the proof yields
\[
N_r\Big( Q(\textbf{\emph{i}}',R) \cap F \Big)\  \leq \  3 \, C_\varepsilon \, c_{\max}^{-2}\,  \alpha_{\min}^{-1}  \, d_{\min}^{-1} \, \, \bigg(\frac{R}{r} \bigg)^{ s +\varepsilon}
\]
which upon letting $\varepsilon \to 0$ gives the desired upper bound. \hfill \qed
\\ \\
\textbf{Lower bound.}  The proof of the lower bound will employ some of the techniques used in Mackay \cite{mackay}.  In particular, we will construct \emph{weak tangents} with the desired dimension.  Weak tangents were used in Section \ref{selfsimexample} to find a lower bound for the dimension of a self-similar set with overlaps.  Here we require the 2 dimensional version, which also follows from \cite[Proposition 7.1.5]{mackaytyson}, which we state here for the benefit of the reader.

\begin{prop}[Mackay-Tyson] \label{weaktang}
Let $X \subset \mathbb{R}^2$ be compact and let $F$ be a compact subset of $X$.  Let $T_k$ be a sequence of similarity maps defined on $\mathbb{R}^2$ and suppose that $T_k(F) \cap X \to_{d_\mathcal{H}} \hat{F}$.  Then $\dim_\text{\emph{A}} \hat{F}   \leq  \dim_\text{\emph{A}} F$.
\end{prop}

The set $\hat{F}$ in the above lemma is called a \emph{weak tangent} to $F$.  We are now ready to prove the lower bound.
\begin{proof}
Let $F$ be a self-affine set in the mixed class.  Without loss of generality we may assume that
\[
\max_{i \in \mathcal{I}} \, \max_{j=1,2} \, \Big(\dim_\text{B} \pi_j(F) \, + \,  \dim_\text{B}  \text{Slice}_{j,i}(F) \Big) = \dim_\text{B} \pi_1(F) +  \text{Slice}_{1,i}(F)
\]
for some $i \in \mathcal{I}$ which we now fix.  Also fix $j \in \mathcal{I}$ with $c_j>d_j$ which we may assume exists as otherwise we are in the horizontal or vertical class, which will be dealt with in the following section.  Let $k \in \mathbb{N}$, let
\[
\textbf{\emph{i}}(k) = (\underbrace{j, j, \dots, j}_{k \text{ times}}, i, i, \dots) \in \mathcal{I}^{\mathbb{N}}
\]
and let $X=[0,c_i^{-1}]\times [0,1]$.  We will consider the sequence of approximate squares $\{Q(\textbf{\emph{i}}(k), d_j^k)\}_k$.  Note that for $k \in \mathbb{N}$, we have $k_2(\textbf{\emph{i}}(k), d_j^k) = k$ and let $ k_1(\textbf{\emph{i}}(k), d_j^k) = k+l(k)$ for $l(k) \in \mathbb{N}$ satisfying
\[
c_j^k c_i^{l(k)+1} < d_j^k \leq c_j^k c_i^{l(k)}.
\]
For each $k \in \mathbb{N}$, let $T_k$ be the unique homothetic similarity on $\mathbb{R}^2$ with similarity ratio $d_j^{-k}$ which maps the approximate square $Q(\textbf{\emph{i}}(k), d_j^k)$ to $[0,d_j^{-k} c_i^k c_j^{l(k)}] \times [0,1]  \subseteq X$, mapping the left vertical side of $Q(\textbf{\emph{i}}(k), d_j^k)$ to $\{0\} \times [0,1]$.
\\ \\
Note that we may take a subsequence of the $T_k$ such that $d_j^{-k} c_i^k c_j^{l(k)} \to 1$.  To see this observe that if $\log (d_j/c_j)/(\log c_i) \in \mathbb{Q}$, then there exists a subsequence where $d_j^{-k} c_i^k c_j^{l(k)} = 1$ for all $k$ and if $\log (d_j/c_j)/(\log c_i) \notin \mathbb{Q}$, then the $d_j^{-k} c_i^k c_j^{l(k)}$ are uniformly distributed on $(1, c_i^{-1})$.  Using this and the fact that $\big(\mathcal{K}(X), d_\mathcal{H}\big)$ is compact, we may extract a subsequence of the $T_k$ for which $T_k(F) \cap X$ converges to a weak tangent $\hat{F} \subseteq X$ and $d_j^{-k} c_i^k c_j^{l(k)} \to 1$.
\begin{lma} \label{weaktangprod}
The weak tangent $\hat{F}$ constructed above contains the set $\pi_1(F) \times \pi_2\big(\text{\emph{Slice}}_{1,i}(F)\big) $.
\end{lma}
\begin{proof}
It suffices to show that $T_k\big(Q(\textbf{\emph{i}}(k), d_j^k) \cap F \big)$ converges to $\pi_1(F) \times \pi_2\big(\text{Slice}_{1,i}(F)\big)$ in the Hausdorff metric.  The IFS $\mathcal{I}$ induces an IFS of similarities on the vertical slice through $\Pi(\textbf{\emph{i}})$ (which is the fixed point of $S_i$).  It is easy to see that the attractor of this IFS is isometric to $\text{Slice}_{1,i}(F)$.  Let $E_{k}$ denote the $l(k)$th level in the construction of $\text{Slice}_{1,i}(F)$ via the induced IFS.  We claim that the set $T_k\big(Q(\textbf{\emph{i}}(k), d_j^k) \cap F \big)$ will never be further away than $1-d_j^{k} c_i^{-k} c_j^{-l(k)} \ + \ d_{\max}^{l(k)}$ from the set $\pi_1(F) \times \pi_2(E_k)$ in the Hausdorff metric.  To see this observe that if we scale $T_k\big(Q(\textbf{\emph{i}}(k), c_j^k) \cap F \big)$ horizontally by $d_j^{k} c_i^{-k} c_j^{-l(k)}$ it becomes a set, $\pi_1(F) \times H_k$ for some set $H_k \subseteq \pi_2(E_k)$ with the property that $H_k$ intersects every basic interval in $\pi_2(E_k)$.  Since each basic interval in $\pi_2(E_k)$ has length no greater than $d_{\max}^{l(k)}$ we have that $\pi_1(F) \times \pi_2(E_k)$ is contained in the $d_{\max}^{l(k)}$ neighbourhood of $\pi_1(F) \times H_k$.  Whence
\begin{eqnarray*}
d_\mathcal{H} \Big(T_k\big(Q(\textbf{\emph{i}}(k), d_j^k) \cap F \big) , \ \pi_1(F) \times \pi_2(E_k) \Big) &\leq& d_\mathcal{H} \Big( T_k\big(Q(\textbf{\emph{i}}(k), d_j^k) \cap F \big)\big) , \ \pi_1(F) \times H_k \Big)  \\ \\ &\,& \qquad  + \ d_\mathcal{H} \Big( \pi_1(F) \times H_k , \ \pi_1(F) \times \pi_2(E_k) \Big) \\ \\
&\leq& 1-d_j^{k} c_i^{-k} c_j^{-l(k)} \ + \ d_{\max}^{l(k)}.
\end{eqnarray*}
It follows from the claim that
\begin{eqnarray*}
d_\mathcal{H} \Big( T_k\big(Q(\textbf{\emph{i}}(k), d_j^k)  \cap F \big) , \ \ \pi_1(F) \times \pi_2\big(\text{Slice}_j(F)\big) \Big) &\leq& d_\mathcal{H} \Big( T_k\big(Q(\textbf{\emph{i}}(k), d_j^k) \cap F \big) ,  \ \ \pi_1(F) \times \pi_2(E_k) \Big)  \\ \\ &\,& \qquad + \ d_\mathcal{H} \Big( \pi_1(F) \times \pi_2(E_k), \ \ \pi_1(F) \times \pi_2\big(\text{Slice}_j(F)\big) \Big) \\ \\
&\leq& \Big(1-d_j^{k} c_i^{-k} c_j^{-l(k)} \ + \ d_{\max}^{l(k)} \Big)  \ + \ d_{\max}^{l(k)}\\ \\
&\to& 0
\end{eqnarray*}
as $k \to \infty$, since
\[
l(k) \ > \  k \frac{\log(d_j/c_j)}{\log c_i} -1 \to \infty \qquad \text{and} \qquad  d_j^{-k} c_i^k c_j^{l(k)} \to 1
\]
which completes the proof of Lemma \ref{weaktangprod}.
\end{proof}
We can now complete the proof of the lower bound by estimating the Assouad dimension of $F$ from below, using the fact that $\hat{F}$ is a product of two self-similar sets.

\begin{eqnarray*}
\dim_\text{B} \pi_1(F) + \dim_\text{B} \text{Slice}_{1,i}(F) &=& \dim_\text{B} \Big(\pi_1(F) \times \pi_2\big(\text{Slice}_{1,i}(F)\big)  \Big)  \\ \\
&\leq& \dim_\text{B} \hat{F} \qquad \text{by Lemma \ref{weaktangprod}} \\ \\
&\leq& \dim_\text{A} \hat{F} \  \leq \  \dim_\text{A} F
\end{eqnarray*}
by Proposition \ref{weaktang}.
\end{proof}

\subsection{Proof of Theorem \ref{uppermain} for the horizontal and vertical classes} \label{upperlalley}

This is similar to the proof in the mixed case and so we only briefly discuss it.
\\ \\
\textbf{Upper bound.} We break up $N_r\big( Q(\textbf{\emph{i}}, R) \cap F  \big)$ in the same way except in this case we may omit either the second or the third term as the smallest singular value always corresponds to either the vertical contraction (in the horizontal class) or the horizontal contraction (in the vertical class).  The rest of the proof proceeds in the same way.
\\ \\
\textbf{Lower bound.}  One can construct a weak tangent with the required dimension.  The key difference to Mackay's argument \cite{mackay} is that, since we may be in the extended Lalley-Gatzouras case, we may not be able to fix a map at the beginning to `follow into the construction'.  Either one can iterate the IFS to find a genuinely affine map which one can `follow in' to find a weak tangent with dimension arbitrarily close to the required dimension, or one can follow our proof in the previous section and choose a genuinely affine map for the first $k$ stages and then switch to a map in the correct column.

\subsection{Proof of Theorem \ref{lowermain} for the mixed class} \label{lowerbaranski}

\textbf{Upper bound.}
Since lower dimension is a natural dual to Assouad dimension and tends to `mirror' the Assouad dimension in many ways, one might expect, given that weak tangents provide a very natural way to find \emph{lower bounds} for Assouad dimension, that weak tangents might provide a way of giving \emph{upper bounds} for lower dimension.  In particular, in light of Proposition \ref{weaktang}, one might na\"ively expect the following statement to be true:
\\ \\
\emph{``Let $X \subset \mathbb{R}^2$ be compact and let $F$ be a compact subset of $X$.  Let $T_k$ be a sequence of similarity maps defined on $\mathbb{R}^2$ and suppose that $T_k(F) \cap X \to_{d_\mathcal{H}} \hat{F}$.  Then $\dim_\text{\emph{L}} \hat{F}   \geq  \dim_\text{\emph{L}} F$.''}
\\ \\
However, it is easy to see that this is false as one can often find weak tangents with isolated points, and hence lower dimension equal to zero, even if the original set has positive lower dimension.  However, we will now state and prove what we believe is the natural analogue of Proposition \ref{weaktang} for the lower dimension.  Note that we also give a slight strengthening of Proposition \ref{weaktang} in that we relax of the conditions on the maps $\{T_k\}$ from \emph{similarity maps} to certain classes of \emph{bi-Lipschitz maps}.  This change is specifically designed to deal with the lower dimension because we can now make the weak tangent precisely equal to the limit of scaled versions of approximate squares.  This is necessary because lower dimension is not monotone and so an analogue of Lemma \ref{weaktangprod} would not suffice.  We include the statement for Assouad dimension for completeness.  The key feature for the lower dimension result is the existence of a constant $\theta \in (0,1]$ with the properties described below.  This is required to prevent the unwanted introduction of isolated points or indeed any points around which the set is inappropriately easy to cover.  We call the `tangents' described in the following proposition \emph{very weak tangents}.
\begin{prop}[very weak tangents] \label{lowerweaktang}
Let $X \subset \mathbb{R}^n$ be compact and let $F$ be a compact subset of $X$.  Let $T_k$ be a sequence of bi-Lipschitz maps defined on $\mathbb{R}^n$ with Lipschitz constants $a_k,b_k \geq 1$ such that
\[
a_k \lvert x-y\rvert\  \leq \ \lvert T_k(x) - T_k(y) \rvert \  \leq \  b_k \lvert x-y \rvert \qquad (x,y \in \mathbb{R}^n)
\]
and
\[
\sup_k \, b_k/a_k \ = \ C_0 \ < \  \infty
\]
and suppose that $T_k(F) \cap X \to_{d_\mathcal{H}} \hat{F}$.  Then
\[
\dim_\text{\emph{A}} \hat{F} \  \leq \   \dim_\text{\emph{A}} F.
\]
If, in addition, there exists a uniform constant $\theta \in (0,1]$ such that for all $r \in (0,1]$ and $\hat{x} \in \hat{F}$, there exists $\hat{y} \in \hat{F}$ such that $B(\hat{y}, r\theta) \subseteq B(\hat{x}, r) \cap X$;  then
\[
\dim_\text{\emph{L}} F \  \leq  \ \dim_\text{\emph{L}} \hat{F} \  \leq \   \dim_\text{\emph{A}} \hat{F} \ \leq \ \dim_\text{\emph{A}} F.
\]
\end{prop}

\begin{proof} Let $F \subseteq X$ be a compact set and assume that $\dim_\text{L} F >0$.  If $\dim_\text{L} F =0$, then the lower estimate is trivial.  Let $\hat{F}$ be a very weak tangent to $F$, as described above, and let $\alpha, \beta \in (0,\infty)$ with $\alpha<  \dim_\text{L} F \leq  \dim_\text{A} F < \beta$.  It follows from the fact that the $T_k$ are bi-Lipschitz maps with Lipschitz constants $b_k \geq a_k  \geq 1$ satisfying $\sup_k \, b_k/a_k \ = \ C_0 \ < \  \infty$ that there exists uniform constants $C_1, C_2, \rho>0$ such that for all $k \in \mathbb{N}$, all $0<r < R \leq \rho$ and all $x \in T_k(F)$ we have
\[
C_1 \, C_0^{-\alpha} \, \bigg(\frac{R}{r}\bigg)^\alpha \ \leq \ N_r\big( B(x,R) \cap T_k(F) \big) \ \leq \ C_2 \, C_0^\beta \, \bigg(\frac{R}{r}\bigg)^\beta.
\]
Fix $0<r < R \leq \rho$ and fix $\hat{x} \in \hat{F}$.  Choose $k \in \mathbb{N}$ such that $d_\mathcal{H}(T_k(F) \cap X , \hat{F}) < r/2$.  It follows that there exists $x \in T_k(F) \cap X$ such that $B(\hat{x}, R) \cap \hat{F} \subseteq B(x, 2R)$ and hence, given any $r/2$-cover of $B(x, 2R) \cap T_k(F)$, we may find an $r$-cover of $B(\hat{x}, R) \cap \hat{F}$ by the same number of sets.  Thus
\[
N_r\big( B(\hat{x}, R) \cap \hat{F} \big) \ \leq \ N_{r/2}\big( B(x, 2R) \cap T_k(F) \big) \ \leq \  C_2 \, C_0^\beta \,\bigg(\frac{2R}{r/2}\bigg)^\beta \ = \ C_2 \, C_0^\beta \, 4^\beta \, \bigg(\frac{R}{r}\bigg)^\beta
\]
which proves that $\dim_\text{A} \hat{F} \  \leq \   \dim_\text{A} F$.
\\ \\
For the lower estimate assume that there exists $\theta \in (0,1]$ satisfying the above property and fix $\hat{x} \in \hat{F}$.  We may thus find $\hat{y} \in \hat{F}$ such that $B(\hat{y}, R\theta) \subseteq B(\hat{x}, R) \cap X$.  Choose $k \in \mathbb{N}$ such that $d_\mathcal{H}(T_k(F) \cap X , \hat{F}) < \min\{r/2,R\theta/2\}$.  It follows that there exists $y \in T_k(F) \cap X$ such that $B(y, R\theta/2) \subseteq B(\hat{y}, R\theta)  \subseteq B(\hat{x}, R) \cap X$ and hence, given any $r$-cover of $B(\hat{x}, R) \cap \hat{F}$, we may find an $2r$-cover of $B(y, R\theta/2) \cap T_k(F) \cap X = B(y, R\theta/2) \cap T_k(F)$ by the same number of sets.  Thus 
\[
N_r\big( B(\hat{x}, R) \cap \hat{F} \big) \ \geq \ N_{2r}\big( B(y, R\theta/2) \cap T_k(F) \big) \ \geq \  C_1 \, C_0^{-\alpha} \,\bigg(\frac{R\theta/2}{2r}\bigg)^\alpha \ = \ C_1 \, C_0^{-\alpha} \, (\theta/4)^\alpha \, \bigg(\frac{R}{r}\bigg)^\alpha
\]
which proves that $\dim_\text{L} \hat{F} \  \geq  \ \dim_\text{L} F$.
\end{proof}

We will now turn to the proof of Theorem \ref{lowermain}.  Let $F$ be a self-affine set in the mixed class.  Without loss of generality we may assume that
\[
\min_{i \in \mathcal{I}} \, \min_{j=1,2} \, \Big(\dim_\text{B} \pi_j(F) \, + \,  \dim_\text{B}  \text{Slice}_{j,i}(F) \Big) = \dim_\text{B} \pi_1(F) +  \text{Slice}_{1,i}(F)
\]
for some $i \in \mathcal{I}$ which we now fix.  Also assume that the column in the construction pattern which contains the rectangle corresponding to $i$ contains at least one other rectangle.  Why we assume this will become clear during the proof and we will deal with the other case afterwards.  Now fix $j \in \mathcal{I}$ with $c_j>d_j$ which we may assume exists as otherwise we are in the horizontal or vertical class, which will be dealt with in the following section.  Let $k \in \mathbb{N}$ and let
\[
\textbf{\emph{i}}(k) = (\underbrace{j, j, \dots, j}_{k \text{ times}}, i, i, \dots) \in \mathcal{I}^{\mathbb{N}}
\]
and let $X=[0,1]^2$.  We will consider the sequence of approximate squares $\{Q(\textbf{\emph{i}}(k), d_j^k)\}_k$. For each $k \in \mathbb{N}$, let $T_k$ be the unique linear bi-Lipschitz map on $\mathbb{R}^2$ which maps the approximate square $Q(\textbf{\emph{i}}(k), d_j^k)$ to $X$, mapping the left vertical side of $Q(\textbf{\emph{i}}(k), d_j^k)$ to $\{0\} \times [0,1]$ and the bottom side of $Q(\textbf{\emph{i}}(k), d_j^k)$ to $[0,1] \times \{0\}$.  Note that this sequence of maps $\{T_k\}$ satisfies the requirements of Proposition \ref{lowerweaktang} with $C_0 = \alpha_{\min}^{-1}$, say.  Since $\big(\mathcal{K}(X), d_\mathcal{H}\big)$ is compact, we may extract a subsequence of the $T_k$ for which $T_k(F) \cap X$ converges to a very weak tangent $\hat{F} \subseteq X$.
\begin{lma} \label{weaktangprod2}
The very weak tangent, $\hat{F}$, constructed above is equal to $\pi_1(F) \times \pi_2\big(\text{\emph{Slice}}_{1,i}(F)\big) $ and, furthermore, there exists $\theta \in (0,1]$ with the desired property from Proposition \ref{lowerweaktang}.
\end{lma}
\begin{proof}
To show that $\hat{F} = \pi_1(F) \times \pi_2\big(\text{Slice}_{1,i}(F)\big) $, it suffices to show that $T_k\big(Q(\textbf{\emph{i}}(k), d_j^k) \cap F \big)$ converges to $\pi_1(F) \times \pi_2\big(\text{Slice}_{1,i}(F)\big)$ in the Hausdorff metric.  This follows by a virtually identical argument to that used in the proof of Lemma \ref{weaktangprod} and is therefore omitted. It remains to show that there exists $\theta \in (0,1]$ such that for all $r \in (0,1]$ and $\hat{x} \in \hat{F}$, there exists $\hat{y} \in \hat{F}$ such that $B(\hat{y}, r\theta) \subseteq B(\hat{x}, r) \cap X$.  We will first prove that the one dimensional analogue of this property holds for self-similar subsets of $[0,1]$.  In particular, let $E \subseteq [0,1]$ be the self-similar attractor of an IFS consisting of $N\geq 2$ homothetic similarities with similarity ratios $\{c_1, \dots, c_N\}$ ordered from left to right by translation vector and write $c_{\min}$ for the smallest contraction ratio.  We will prove that there exists $\theta \in (0,1]$ such that for all $r \in (0,1]$ and $x \in E$, there exists $y \in E$ such that $B(y, r\theta) \subseteq B(x, r) \cap [0,1]$. If $E \subseteq (0,1)$, then we may choose $\theta = \inf_{x \in E, y=0,1} \lvert x-y\rvert>0$ and then for any $x \in E$, we may choose $y=x$.  Thus we assume without loss of generality that $0 \in E$ and so $c_1$ is the contraction ratio of a map which fixes 0.  Also, write $z = \sup_{x \in E} \lvert x \rvert$.  It suffices to prove the result in the case $x=0$ and $r \in (0,z]$.  Observe that $c_1^kz \in E$ for all $k \in \mathbb{N}_0$ and let
\[
k = \min \big\{l \in \mathbb{N}_0 : c_1^l z < r(1-c_{\min})\big\},
\]
$y = c_1^k z \in E$ and $\theta = c_{\min}(1-c_{\min})$. To see that this choice of $y$ and $\theta$ works, observe that
\[
y+\theta r = c_1^k z  + c_{\min}(1-c_{\min})r <  r(1-c_{\min}) + c_{\min}r = r
\]
and
\[
y-\theta r = c_1^k z  - c_{\min}(1-c_{\min})r \geq c_1  r(1-c_{\min}) - c_{\min}(1-c_{\min})r \geq  r(1-c_{\min})(c_1-c_{\min}) \geq 0
\]
and so $B(y, r\theta) \subseteq B(x, r) \cap [0,1] = [0,r)$.  Finally, observe that our set $\hat F$ is the product of two self-similar sets, $E_1$ and $E_2$, of the above form with constants $\theta_1$ and $\theta_2$ giving the desired one dimensional property.  Now let $r \in (0,1]$ and $\hat{x}=(x_1,x_2) \in \hat{F} = E_1 \times E_2$.  By the above argument, there exists $y_1 \in E_1$ and $y_2 \in E_2$ such that $B(y_1, r\theta_1) \subseteq B(x_1, r) \cap [0,1]$ and $B(y_2, r\theta_2) \subseteq B(x_2, r) \cap [0,1]$.  Setting $\hat y = (y_1,y_2) \in \hat F$, it follows that $B\big(\hat y, r \min(\theta_1,\theta_2)\big) \subseteq B(\hat x, r) \cap [0,1]^2$, which completes the proof.
\end{proof}
We can now complete the proof of the upper bound by estimating the lower dimension of $F$ from above using the fact that $\hat{F}$ is a very weak tangent to $F$ and is the product of two self-similar sets.  We have
\begin{eqnarray*}
 \dim_\text{B} \pi_1(F) + \dim_\text{B} \text{Slice}_{1,i}(F) &=& \dim_\text{B} \Big(\pi_1(F) \times \pi_2\big(\text{Slice}_{1,i}(F)\big)  \Big)  \\ \\
&=& \dim_\text{B} \hat{F} \qquad \text{by Lemma \ref{weaktangprod2}} \\ \\
&\geq& \dim_\text{L} \hat{F} \  \geq \  \dim_\text{L} F
\end{eqnarray*}
by Proposition \ref{lowerweaktang}.  Finally, we have to deal with the case where $i$ corresponds to a rectangle which is alone in some column in the construction.  In this case we can construct a very weak tangent to $F$ as above, but we may not be able to find a constant $\theta$ with the desired properties.  In particular, if the rectangle corresponding to $i$ is at the top or bottom of the column, then the very weak tangent will lie on the boundary of $X$.  However, this problem is easy to overcome.  Let $\varepsilon>0$ and note that by iterating the IFS we may produce a new IFS, $\mathcal{I}'$, with the same attractor which has some $i' \in \mathcal{I}'$ for which $ \dim_\text{B} \pi_1(F) +  \text{Slice}_{1,i'}(F)< \dim_\text{B} \pi_1(F) +  \text{Slice}_{1,i}(F) +\varepsilon$ and does not correspond to a rectangle which is in a column by itself.  We can then construct a very weak tangent to $F$ in the above manor with dimension $\varepsilon$-close to the desired dimension which is sufficient to complete the proof of the upper bound.
\\ \\
\textbf{Lower bound.}  The following proof is in the same spirit as the proof of the \emph{upper} bound in Theorem \ref{uppermain}.   Fix $\textbf{\emph{i}}' \in \mathcal{I}^\mathbb{N}$, $R>0$ and $r \in (0, R)$ and as before we will consider the approximate square $Q(\textbf{\emph{i}}',R)$.  Without loss of generality assume that $k_1(\textbf{\emph{i}}',R) \geq k_2(\textbf{\emph{i}}',R)$ and let $k=k_1(\textbf{\emph{i}}',R)$.   Furthermore we may assume that there exists $j_1, j_2 \in \mathcal{I}$ such that $c_{j_1} > d_{j_1}$ and $c_{j_2} < d_{j_2}$ as otherwise we are not in the mixed class. Let
\[
s = \min_{i \in \mathcal{I}} \, \min_{j=1,2} \, \Big(\dim_\text{B} \pi_j(F) \, + \,  \dim_\text{B}  \text{Slice}_{j,i}(F) \Big).
\]
It suffices to prove that for all $\varepsilon \in (0,1)$, there exists a constant $C(\varepsilon)$ such that
\[
N_r\Big( Q(\textbf{\emph{i}}',R) \cap F  \Big) \ \geq \ C(\varepsilon) \, \bigg(\frac{R}{r} \bigg)^{ s -\varepsilon}.
\]
Let $\varepsilon \in (0,1)$.  As before, writing
\[
\mathcal{I}_Q = \mathcal{I}_{Q(\textbf{\emph{i}}',R)} = \big\{ \textbf{\emph{j}} \in \mathcal{I}^k: S_\textbf{\emph{j}}(F) \subseteq Q(\textbf{\emph{i}}',R) \big\},
\]
and
\[
\mathcal{I}_Q^{<r} = \big\{\textbf{\emph{i}} \in \mathcal{I}_Q : d_\textbf{\emph{i}} < r \big\},
\]
we have
\begin{eqnarray*}
N_r\Big( Q(\textbf{\emph{i}}',R)  \cap F  \Big) &=& N_r\Bigg( \ \ \bigcup_{\substack{\textbf{\emph{i}} \in \mathcal{I}_Q^{<r}}}  S_{\textbf{\emph{i}}}(F) \ \ \cup \ \ \bigcup_{\substack{\textbf{\emph{i}} \in \mathcal{I}_Q: \\ \\
d_\textbf{\emph{i}} \geq r}}  S_{\textbf{\emph{i}}}(F)   \  \ \Bigg).
\end{eqnarray*}
At this point in the proof of the upper bound in Theorem \ref{uppermain}, we iterated the IFS $\mathcal{I}$ within each of the sets $\{  S_{\textbf{\emph{i}}}(F) : \textbf{\emph{i}} \in \mathcal{I}_Q \text{ s.t. }
d_\textbf{\emph{i}} \geq r\}$ to decompose $F$ into small `rectangular parts' with smallest side comparable to $r$.  If we proceed in this way here, then the `third term' causes problems.  In particular, we end up with a term containing
\[
 \sum_{\substack{\textbf{\emph{i}} \in \mathcal{I}_Q}} (d_\textbf{\emph{i}}/R)^{s_2-\varepsilon}
\]
which we wish to bound from below, but cannot as $s_2-\varepsilon$ may be too large.  This problem does not occur in the proof of the upper bound in Theorem \ref{uppermain} as the term
\[
 \sum_{\substack{\textbf{\emph{i}} \in \mathcal{I}_Q}} (d_\textbf{\emph{i}}/R)^{s_2+\varepsilon}
\]
\emph{can} be bounded from \emph{above}.  This was a surprising and interesting complication.  To overcome this we need to engineer it so that the third term disappears.  As such we will iterate only using maps $S_\textbf{\emph{j}}$ which have $c_\textbf{\emph{j}} \geq d_\textbf{\emph{j}}$.  Fortunately we are able to do this by introducing a new IFS $\mathcal{I}_{\varepsilon}$ with the properties outlined in the following lemma.
\begin{lma} \label{Hexists}
There exists an IFS $\mathcal{I}_{\varepsilon}$ of affine maps on $[0,1]^2$ with attractor $F_\varepsilon$ which has the following properties:
\begin{itemize}
\item[(1)] $\mathcal{I}_{\varepsilon}$ is of horizontal type, i.e., $c_\textbf{i} \geq d_\textbf{i}$ for all $\textbf{i} \in \mathcal{I}_{\varepsilon}$,
\item[(2)] $\mathcal{I}_{\varepsilon}$ is a subset of some stopping $\mathcal{I}'$ created from the original IFS $\mathcal{I}$,
\item[(3)] $F_{\varepsilon}$ is a subset of $F$ and is such that $\dim_{\text{\emph{B}}} F_{\varepsilon} \geq s_1+u_1-\varepsilon$.
\end{itemize}
\end{lma}
\begin{proof}
Let 
\[
\mathcal{I}_{0} = \{\textbf{\emph{i}} \in \mathcal{I}^* : c_\textbf{\emph{i}} \geq d_\textbf{\emph{i}} \text{ and } \nexists  \textbf{\emph{j}} \prec \textbf{\emph{i}} \text{ s.t. } \textbf{\emph{j}} \neq  \textbf{\emph{i}} \text{ and } c_\textbf{\emph{j}} \geq d_\textbf{\emph{j}} \}
\]
and let 
\[
\mathcal{I}_{k} = \{\textbf{\emph{i}} \in \mathcal{I}_0 : \lvert \textbf{\emph{i}}\rvert \leq k \}.
\]
It is clear that $\mathcal{I}_{k}$ satisfies properties (1) and (2) and that $k$ can be chosen large enough to ensure that property (3) is satisfied.
\end{proof}

We treat $\mathcal{I}_\varepsilon$ like $\mathcal{I}$ and write $(\mathcal{I}_\varepsilon)^* = \bigcup_{k\geq1} (\mathcal{I}_\varepsilon)^k$ to denote the set of all finite sequences with entries in $\mathcal{I}_\varepsilon$ and
\[
\alpha_{\varepsilon, \min} = \min \{\alpha_2(\textbf{\emph{i}}) : \textbf{\emph{i}} \in \mathcal{I}_\varepsilon \}>0,
\]
which clearly depends on $\varepsilon$.  We have

\begin{eqnarray*}
N_r\Big( Q(\textbf{\emph{i}}',R)  \cap F  \Big) &=& N_r\Bigg( \ \ \bigcup_{\substack{\textbf{\emph{j}} \in \mathcal{I}_r :\\ \\
 \exists \textbf{\emph{i}} \in \mathcal{I}_Q^{<r}, \\ \\
\textbf{\emph{j}} \prec \textbf{\emph{i}}  }}  \Big(S_{\textbf{\emph{j}}}(F) \cap Q(\textbf{\emph{i}}',R)\Big) \ \ \cup \ \ \bigcup_{\substack{\textbf{\emph{i}} \in \mathcal{I}_Q: \\ \\
d_\textbf{\emph{i}} \geq r}}  \ \bigcup_{\substack{\textbf{\emph{j}} \in (\mathcal{I}_{\varepsilon})^*:\\ \\
\alpha_2(\textbf{\emph{ij}})<r \leq \alpha_2(\textbf{\emph{i}}\overline{\textbf{\emph{j}}})} } S_{\textbf{\emph{ij}}}(F)   \  \ \Bigg).
\end{eqnarray*}
Let $U$ be any $r \times r$ closed square with sides parallel to the coordinate axes and let
\[
M_\varepsilon = \min\{ n \in \mathbb{N} : n \geq \alpha_{\varepsilon, \min}^{-1}+2\}.
\]
Observe that each of the sets $S_{\textbf{\emph{j}}}(F) \cap Q(\textbf{\emph{i}}',R)$ and $S_{\textbf{\emph{ij}}}(F)$ inside the above unions lies in a rectangle whose smallest side is of length at least $\alpha_{\varepsilon, \min} r$ and the interiors of these rectangles are pairwise disjoint.  It follows from this that $U$ can intersect no more than $M_\varepsilon^2$ of the sets $S_{\textbf{\emph{j}}}(F) \cap Q(\textbf{\emph{i}}',R)$ and $S_{\textbf{\emph{ij}}}(F)$.  Whence, using the $r$-grid definition of $N_r$,
\begin{eqnarray*}
M_\varepsilon^2 \, N_r\Big( Q(\textbf{\emph{i}}',R)  \cap F  \Big) &\geq& \sum_{\substack{\textbf{\emph{j}} \in \mathcal{I}_r :\\ \\
 \exists \textbf{\emph{i}}\in \mathcal{I}_Q^{<r}, \\ \\
\textbf{\emph{j}} \prec \textbf{\emph{i}}  }} N_r\Big(   S_{\textbf{\emph{j}}}(F) \cap Q(\textbf{\emph{i}}',R)  \Big) \ \ + \ \ \sum_{\substack{\textbf{\emph{i}} \in \mathcal{I}_Q: \\ \\
d_\textbf{\emph{i}} \geq r}}  \ \sum_{\substack{\textbf{\emph{j}} \in (\mathcal{I}_{\varepsilon})^*:\\ \\
\alpha_2(\textbf{\emph{ij}})<r \leq \alpha_2(\textbf{\emph{i}}\overline{\textbf{\emph{j}}})} } N_r\Big( S_{\textbf{\emph{ij}}}(F)   \Big)\\ \\
&\geq& \sum_{\substack{\textbf{\emph{j}} \in \mathcal{I}_r :\\ \\
 \exists \textbf{\emph{i}}\in \mathcal{I}_Q^{<r}, \\ \\
\textbf{\emph{j}} \prec \textbf{\emph{i}}  }} N_{r/c_\textbf{\emph{i}}}(\pi_1(F)) \ + \  \sum_{\substack{\textbf{\emph{i}} \in \mathcal{I}_Q: \\ \\
d_\textbf{\emph{i}} \geq r}}  \ \sum_{\substack{\textbf{\emph{j}} \in (\mathcal{I}_{\varepsilon})^*:\\ \\
\alpha_2(\textbf{\emph{ij}})<r \leq \alpha_2(\textbf{\emph{i}}\overline{\textbf{\emph{j}}})} } N_{r/c_\textbf{\emph{ij}}} \big( \pi_1(F)  \big)
\end{eqnarray*}

As before, we will analyse each of the above terms separately.  For the first term we have
\begin{eqnarray*}
\sum_{\substack{\textbf{\emph{j}} \in \mathcal{I}_r :\\ \\
 \exists \textbf{\emph{i}}\in \mathcal{I}_Q^{<r}, \\ \\
\textbf{\emph{j}} \prec \textbf{\emph{i}}  }} N_{r/c_\textbf{\emph{i}}}(\pi_1(F))
&\geq& \sum_{\substack{\textbf{\emph{j}} \in \mathcal{I}_r :\\ \\
 \exists \textbf{\emph{i}}\in \mathcal{I}_Q^{<r}, \\ \\
\textbf{\emph{j}} \prec \textbf{\emph{i}}  }} \tfrac{1}{C_\varepsilon}  \, \bigg( \frac{R}{r} \bigg)^{s_1-\varepsilon}  \qquad \text{by (\ref{simplebox1}) and Lemma \ref{cubes} (1)} \\ \\
&\geq&  \tfrac{1}{C_\varepsilon}  \,\bigg( \frac{ R}{r} \bigg)^{s_1+u_1-\varepsilon} \sum_{\substack{\textbf{\emph{j}} \in \mathcal{I}_r :\\ \\
 \exists \textbf{\emph{i}}\in \mathcal{I}_Q^{<r}, \\ \\
\textbf{\emph{j}} \prec \textbf{\emph{i}}  }} (r/R)^{u_1}  \\ \\
&\geq&  \tfrac{1}{C_\varepsilon}  \,\bigg( \frac{ R}{r} \bigg)^{s-\varepsilon} \sum_{\substack{\textbf{\emph{j}} \in \mathcal{I}_r :\\ \\
 \exists \textbf{\emph{i}}\in \mathcal{I}_Q^{<r}, \\ \\
\textbf{\emph{j}} \prec \textbf{\emph{i}}  }} (d_\textbf{\emph{j}}/R)^{u_1}  \qquad \text{by (\ref{stop1})}
\end{eqnarray*}

For the second term we have
\begin{eqnarray*}
 \sum_{\substack{\textbf{\emph{i}} \in \mathcal{I}_Q: \\ \\
d_\textbf{\emph{i}} \geq r}}  \ \sum_{\substack{\textbf{\emph{j}} \in (\mathcal{I}_{\varepsilon})^*:\\ \\
\alpha_2(\textbf{\emph{ij}})<r \leq \alpha_2(\textbf{\emph{i}}\overline{\textbf{\emph{j}}})} } N_{r/c_\textbf{\emph{ij}}} \big( \pi_1(F)  \big)
&\geq&  \sum_{\substack{\textbf{\emph{i}} \in \mathcal{I}_Q: \\ \\
d_\textbf{\emph{i}} \geq r}}  \ \sum_{\substack{\textbf{\emph{j}} \in (\mathcal{I}_{\varepsilon})^*:\\ \\
\alpha_2(\textbf{\emph{ij}})<r \leq \alpha_2(\textbf{\emph{i}}\overline{\textbf{\emph{j}}})} } \tfrac{1}{C_\varepsilon} \, \Big(\frac{c_\textbf{\emph{i}}c_\textbf{\emph{j}}}{r} \Big)^{s_1-\varepsilon}
 \qquad \text{by (\ref{simplebox1})} \\ \\
&\geq&  \tfrac{1}{C_\varepsilon} \,\bigg(\frac{1}{r} \bigg)^{s_1-\varepsilon} \sum_{\substack{\textbf{\emph{i}} \in \mathcal{I}_Q:\\ \\
d_\textbf{\emph{i}} \geq r}}  \ c_\textbf{\emph{i}}^{s_1-\varepsilon} \,  \sum_{\substack{\textbf{\emph{j}} \in (\mathcal{I}_{\varepsilon})^*:\\ \\
\alpha_2(\textbf{\emph{ij}})<r \leq \alpha_2(\textbf{\emph{i}}\overline{\textbf{\emph{j}}})} }  c_\textbf{\emph{j}}^{s_1} \\ \\
&\geq& \tfrac{1}{C_\varepsilon} \,\bigg(\frac{R}{r} \bigg)^{s_1-\varepsilon} \sum_{\substack{\textbf{\emph{i}} \in \mathcal{I}_Q: \\ \\
d_\textbf{\emph{i}} \geq r}}  \  \sum_{\substack{\textbf{\emph{j}} \in (\mathcal{I}_{\varepsilon})^*:\\ \\
\alpha_2(\textbf{\emph{ij}})<r \leq \alpha_2(\textbf{\emph{i}}\overline{\textbf{\emph{j}}})} } c_\textbf{\emph{j}}^{s_1}  \big( d_\textbf{\emph{i}}d_\textbf{\emph{j}} r^{-1} \big)^{u_1} \\ \\
&\quad& \qquad  \qquad \text{by Lemma \ref{cubes} (1) and since $r>\alpha_2(\textbf{\emph{ij}}) = d_\textbf{\emph{i}}d_\textbf{\emph{j}}$}\\ \\
&\geq&  \tfrac{1}{C_\varepsilon} \,\bigg(\frac{R}{r} \bigg)^{s_1-\varepsilon} \,\bigg(\frac{1}{r} \bigg)^{u_1} \sum_{\substack{\textbf{\emph{i}} \in \mathcal{I}_Q: \\ \\
d_\textbf{\emph{i}} \geq r}}  \ d_\textbf{\emph{i}}^{u_1}  \sum_{\substack{\textbf{\emph{j}} \in (\mathcal{I}_{\varepsilon})^*:\\ \\
\alpha_2(\textbf{\emph{ij}})<r \leq \alpha_2(\textbf{\emph{i}}\overline{\textbf{\emph{j}}})} }  c_\textbf{\emph{j}}^{s_1}    d_\textbf{\emph{j}}^{u_1}\\ \\
&\geq& \tfrac{1}{C_\varepsilon} \,\bigg(\frac{R}{r} \bigg)^{s_1+u_1-\varepsilon} \, \sum_{\substack{\textbf{\emph{i}} \in \mathcal{I}_Q: \\ \\
d_\textbf{\emph{i}} \geq r}}  \ (d_\textbf{\emph{i}}/R)^{u_1}  \sum_{\substack{\textbf{\emph{j}} \in (\mathcal{I}_{\varepsilon})^*:\\ \\
\alpha_2(\textbf{\emph{ij}})<r \leq \alpha_2(\textbf{\emph{i}}\overline{\textbf{\emph{j}}})} }  c_\textbf{\emph{j}}^{s_1}    d_\textbf{\emph{j}}^{\dim_\text{B}F_\varepsilon+\varepsilon -s_1}\\ \\
&\quad& \qquad  \qquad \text{by Lemma \ref{DADB}}\\ \\
&\geq& \tfrac{1}{C_\varepsilon} \,\bigg(\frac{R}{r} \bigg)^{s_1+u_1-\varepsilon} \, \sum_{\substack{\textbf{\emph{i}} \in \mathcal{I}_Q: \\ \\
d_\textbf{\emph{i}} \geq r}}  \ (d_\textbf{\emph{i}}/R)^{u_1} \ d_\textbf{\emph{j}}^\varepsilon \sum_{\substack{\textbf{\emph{j}} \in (\mathcal{I}_{\varepsilon})^*:\\ \\
\alpha_2(\textbf{\emph{ij}})<r \leq \alpha_2(\textbf{\emph{i}}\overline{\textbf{\emph{j}}})} }  c_\textbf{\emph{j}}^{s_1}    d_\textbf{\emph{j}}^{\dim_\text{B}F_\varepsilon-s_1}\\ \\
&\geq&  \tfrac{1}{C_\varepsilon} \, \alpha_{\varepsilon, \min} \, \bigg(\frac{R}{r} \bigg)^{s-2\varepsilon} \, \sum_{\substack{\textbf{\emph{i}} \in \mathcal{I}_Q: \\ \\
d_\textbf{\emph{i}} \geq r}}  \ (d_\textbf{\emph{i}}/R)^{u_1}
\end{eqnarray*}
by Lemma \ref{DADBadditive} and the fact that $d_\textbf{\emph{j}} \geq (r/d_\textbf{\emph{i}}) \,\alpha_{\varepsilon, \min} \geq  (r/R) \,\alpha_{\varepsilon, \min}$.  Combining the estimates for the two terms introduced above yields
\begin{eqnarray*}
M_\varepsilon^2 \, N_r\Big( Q(\textbf{\emph{i}}',R)  \cap F  \Big) &\geq&   \tfrac{1}{C_\varepsilon}  \,\bigg( \frac{ R}{r} \bigg)^{s-\varepsilon} \sum_{\substack{\textbf{\emph{j}} \in \mathcal{I}_r :\\ \\
 \exists \textbf{\emph{i}}\in \mathcal{I}_Q^{<r}, \\ \\
\textbf{\emph{j}} \prec \textbf{\emph{i}}  }} (d_\textbf{\emph{j}}/R)^{u_1} \ + \  \tfrac{1}{C_\varepsilon} \,  \, \alpha_{\varepsilon, \min} \bigg(\frac{R}{r} \bigg)^{s-2\varepsilon} \, \sum_{\substack{\textbf{\emph{i}} \in \mathcal{I}_Q: \\ \\
d_\textbf{\emph{i}} \geq r}}  \ (d_\textbf{\emph{i}}/R)^{u_1} \\ \\
&\geq&  \tfrac{1}{C_\varepsilon} \, \alpha_{\varepsilon, \min} \,\bigg( \frac{ R}{r} \bigg)^{s-2\varepsilon} \sum_{\textbf{\emph{j}} \in \mathcal{I}_0} (d_\textbf{\emph{j}}/R)^{u_1} 
\end{eqnarray*}
where
\[
\mathcal{I}_0 \ := \  \{\textbf{\emph{j}} \in \mathcal{I}_r : \exists \textbf{\emph{i}}\in \mathcal{I}_Q^{<r} \  \text{s.t.} \ 
\textbf{\emph{j}} \prec \textbf{\emph{i}}\}  \ \cup \  \{\textbf{\emph{i}} \in \mathcal{I}_Q :
d_\textbf{\emph{i}} \geq r \}.
\]
Observe that $\mathcal{I}_0$ is a $Q(\textbf{\emph{i}}',R)$-pseudo stopping and so by Lemma  \ref{selfsim} we have
\[
\sum_{\textbf{\emph{j}} \in \mathcal{I}_0} (d_\textbf{\emph{j}}/R)^{u_1}  \geq 1
\]
which yields
\[
N_r\Big( Q(\textbf{\emph{i}}',R)  \cap F  \Big) \ \geq \  \tfrac{1}{M_\varepsilon^2} \, \tfrac{1}{C_\varepsilon} \, \alpha_{\varepsilon, \min} \,\bigg( \frac{ R}{r} \bigg)^{s-\varepsilon}.
\]
It follows that $\dim_\text{L} F \geq s-2\varepsilon$ and letting $\varepsilon \to 0$ completes the proof.  \hfill \qed

\subsection{Proof of Theorem \ref{lowermain} for the horizontal and vertical classes} \label{lowerlalley}

This is similar to the proof in the mixed case and so we only briefly discuss it.
\\ \\
\textbf{Upper bound.}  As in the mixed class, one can construct a \emph{lower weak tangent} with the required dimension.  The proof is slightly simpler in that for the horizontal class, for example, we necessarily have that $s = \dim_\text{B} \pi_1(F) +  \text{Slice}_{1,i}(F)$ for some $i \in \mathcal{I}$ and that there exists $j \in \mathcal{I}$ with $c_j>d_j$.
\\ \\
\textbf{Lower bound.} The proof is greatly simplified in this case because we do not have the added complication eluded to above.  In particular, we do not have to introduce the `horizontal subsystem' $\mathcal{I}_\varepsilon$ and we can just iterate using $\mathcal{I}$ as before with no third term appearing.

\subsection{Proof of Corollary \ref{intcor1}} \label{proofintcor1}

In this section we will rely on some results from \cite{lalley-gatz}, which technically speaking were not proved in the extended Lalley-Gatzouras case.  However, it is easy to see that their arguments can be extended to cover this situation and give the results we require.  Without loss of generality, let $F$ be a self-affine attractor of an IFS $\mathcal{I}$ in the horizontal class and assume that $\dim_\text{L} F < \dim_\text{A} F$.  It follows from Theorems \ref{uppermain} and \ref{lowermain} that 
\[
\min_{i \in \mathcal{I}} \ \dim_\text{B} \text{Slice}_{1,i}(F) < \max_{i \in \mathcal{I}} \ \dim_\text{B} \text{Slice}_{1,i}(F)
\]
which means that we do not have \emph{uniform vertical fibres} and it follows from \cite{lalley-gatz} that $\dim_\text{H} F < \dim_\text{B} F$.  We will now show that $\dim_\text{L} F < \dim_\text{H} F$.  We will use the formula for the Hausdorff dimension given in \cite{lalley-gatz} and so we must briefly introduce some notation.  Suppose we have $m$ non-empty columns in the construction and we have chosen $n_i$ rectangles from the $i$th column.  For the $j$th rectangle in the $i$th column write $c_i$ for the length of the base and $d_{ij}$ for the height.  Notice that the length of the base depends only on which column we are in.  Then the Hausdorff dimension of $F$ is given by
\[
\dim_\H F = \max \Bigg\{\frac{\sum_i \sum_j p_{ij}\log p_{ij}}{\sum_i \sum_j p_{ij}\log d_{ij}} + \sum_i q_i \log q_i \Bigg(\frac{1}{\sum_i q_i \log c_i}-\frac{1}{\sum_i \sum_j p_{ij}\log d_{ij}} \Bigg) \Bigg\}
\]
where the maximum is taken over all associated probability distributions $\{p_{ij}\}$ on the set $\big\{(i,j) : i \in \{1, \dots, m\}, j \in \{1, \dots, n_i\}\big\}$ and $q_i = \sum_j p_{ij}$.  Notice that this formula may be rewritten as
\[
\max \Bigg\{\frac{\sum_i \sum_j p_{ij}\log \big(q_i/p_{ij}\big)}{\sum_i \sum_j p_{ij}\log (d_{ij})^{-1}} + \frac{\sum_i q_i \log q_i }{\sum_i q_i \log c_i} \Bigg\}
\]
which clearly demonstrates that if we continuously decrease a particular $d_{ij}$, then the Hausdorff dimension continuously decreases.  Note that we may continuously decrease any particular $d_{ij}$ without affecting any other rectangle in the construction.
\\ \\
If $\min_{i \in \mathcal{I}} \ \dim_\text{B} \text{Slice}_{1,i}(F) = 0$, then the result is clear.  However, if 
$\min_{i \in \mathcal{I}} \ \dim_\text{B} \text{Slice}_{1,i}(F) > 0$, then, although we have already noted that $F$ does not have uniform horizontal fibres, we may continuously decrease the $d_{ij}$ to obtain a new IFS $\mathcal{I}_1$, with the same number of rectangles and the same base lengths, which has an attractor $F_1$ where $\dim_{\text{B}}\text{Slice}_{1,j}(F_1) = \min_{i \in \mathcal{I}} \ \dim_\text{B} \text{Slice}_{1,i}(F)$ for each $j \in \mathcal{I}_1$.  It follows from the above argument and Theorems \ref{uppermain} and \ref{lowermain} that
\[
\dim_\text{L} F = \dim_\text{H} F_1 <  \dim_\text{H} F.
\]
It remains to show that $ \dim_\text{B} F < \dim_\text{A} F$.  However, this follows from a dual argument observing that the box dimension of $F$ is given by the unique solution $s$ of
\[
\sum_{i=1}^m \sum_{j=1}^{n_i} c_i^{s_1} d_{ij}^{s-s_1} = 1
\]
(see \cite{lalley-gatz} for the basic case or \cite{fengaffine, me_box} for the extended case) and so we may continuously \emph{increase} the $d_{ij}$ independently (while keeping $d_{ij} \leq c_i$) and, if necessary, add new maps in to certain columns, to form a new construction $F_2$ with uniform vertical fibres and such that
\[
\dim_\text{B} F <\dim_\text{B} F_2 = \dim_\text{A} F.
\]

\subsection{Proof of Corollary \ref{intcor2}}  \label{proofintcor2}

Let $F$ be in the mixed class.  The result for the horizontal and vertical classes follows from Corollary \ref{intcor1}.  Suppose $F$ is such that $\dim_\text{L} F = \dim_\text{B} F$.  It follows from Lemma \ref{DADB} that $D_A = D_B =\dim_\text{L} F \leq s_j+\dim_\text{B} \text{Slice}_{j,i} (F)$ for all $j \in \{1,2\}$ and $i \in \mathcal{I}$.  Whence, using the notation from the proof of Lemma \ref{DADB},
\[
1 \ =  \ \sum_{i \in \mathcal{I}} c_i^{s_1} d_i^{D_A-s_1} \  \geq \  \sum_{i=1}^{m} \hat{c}_i^{s_1} \sum_{j \in \mathcal{C}_i} d_j^{\dim_\text{B} \text{Slice}_{1,j} (F)}  \ = \ 1
\]
and
\[
1  \ =  \ \sum_{i \in \mathcal{I}} d_i^{s_2} c_i^{D_B-s_2}  \ \geq \  \sum_{i=1}^{n} \hat{d}_i^{s_2} \sum_{j \in \mathcal{R}_i} c_j^{\dim_\text{B} \text{Slice}_{2,j} (F)}  \ = \ 1
\]
and so we have equality throughout in the above two lines which implies that
\[
D_A = D_B = \dim_\text{L} F  =     \max_{i \in \mathcal{I}} \, \max_{j=1,2} \, \Big(\dim_\text{B} \pi_j(F) \, + \,  \dim_\text{B}  \text{Slice}_{j,i}(F) \Big) = \dim_\text{A} F,
\]
which completes the proof.  We remark here that the key reason that a symmetric argument cannot be used to show that if $\dim_\text{A} F = \dim_\text{B} F$, then  $\dim_\text{A} F  = \dim_\text{L} F$, is that $\dim_\text{A} F = \dim_\text{B} F$ only implies that $\max\{D_A, D_B\} \geq s_j+\dim_\text{B} \text{Slice}_{j,i} (F)$ for all $j \in \{1,2\}$ and $i \in \mathcal{I}$.  Indeed such an implication is not true, as shown by the example in Section \ref{examplesB}.

\vspace{6mm}

\begin{centering}

\textbf{Acknowledgements}

\end{centering}

The author was supported by an EPSRC Doctoral Training Grant.  He would also like to thank Kenneth Falconer, James Hyde, Thomas Jordan, Tuomas Orponen and Tuomas Sahlsten for some helpful discussions of this work.

\end{document}